\documentclass[11pt]{article}

\usepackage[title]{appendix}

\usepackage{amssymb}
\usepackage{amsthm}
\usepackage{amsmath,amsfonts,amssymb}

\allowdisplaybreaks
\usepackage[mathscr]{eucal}
\usepackage{exscale}
\usepackage[square,sort,comma,numbers]{natbib}
\setcitestyle{authoryear,round}

\usepackage{bm}
\usepackage{eqlist}
\usepackage[dvipsnames]{color}
\usepackage[Lenny]{fncychap}
\usepackage{multirow}
\usepackage{graphicx}
\usepackage{algpseudocode}
\usepackage{algorithm}
\usepackage{caption}
\usepackage{hyperref}
\usepackage{framed}
\usepackage{color}
\usepackage{booktabs}

\usepackage{framed}
\usepackage{algpseudocode}
\usepackage{indentfirst} 
\usepackage{longtable}

\usepackage{tabu}
\usepackage{tikz}

\usepackage{chngpage}
\usepackage{makecell}
\usepackage{threeparttable}

\usepackage[utf8]{inputenc} % this is needed for umlauts
\usepackage[english]{babel} % this is needed for umlauts
\usepackage[T1]{fontenc}    % this is needed for correct output of umlauts in pdf
\usepackage{amssymb,amsmath,amsfonts} % nice math rendering
\usepackage{braket} % needed for \Set
\usepackage{caption}
\usepackage{algorithm}

\hypersetup{
	colorlinks=true,
	linkcolor=blue,
	citecolor=blue,
	citebordercolor={1 1 0},
}

\newtheorem{theorem}{Theorem}[section]

\newtheorem{proposition}[theorem]{Proposition}
\newtheorem{corollary}[theorem]{Corollary}

\newtheorem{remark}[theorem]{Remark}

\font\bigbf=cmbx10 scaled \magstep3

\begin{document}
	
	\title{\bigbf  Exploring the drive-by sensing power of bus fleet through active scheduling}

	\author{Zhuang Dai
		\quad Ke Han$\thanks{Corresponding author, e-mail: kehan@swjtu.edu.cn;}$\\\\
		\textit{\small Institute of System Science and Engineering, School of Transportation and Logistics,}\\
		\textit{\small Southwest Jiaotong University, Chengdu 611756, China}
	}
	
	\maketitle

	\begin{abstract}
		Vehicle-based mobile sensing (a.k.a drive-by sensing) is an important means of surveying urban environment by leveraging the mobility of public or private transport vehicles. Buses, for their extensive spatial coverage and reliable operations, have received much attention in drive-by sensing. Existing studies have focused on the assignment of sensors to a set of lines or buses with no operational intervention, which is typically formulated as set covering or subset selection problems. This paper aims to boost the sensing power of bus fleets through active scheduling, by allowing instrumented buses to circulate across multiple lines to deliver optimal sensing outcome. We consider a fleet consisting of instrumented and normal buses, and jointly optimize sensor assignment, bus dispatch, and intra- or inter-line relocations, with the objectives of maximizing sensing quality and minimizing operational costs, while serving all timetabled trips. By making general assumptions on the sensing utility function, we formulate the problem as a nonlinear integer program based on a time-expanded network. A batch scheduling algorithm is developed following linearization techniques to solve the problem efficiently, which is tested in a real-world case study in Chengdu, China. The results show that the proposed scheme can improve the sensing objective by 12.0\%-20.5\% (single-line scheduling) and 16.3\%-32.1\% (multi-line scheduling), respectively, while managing to save operational costs by 1.0\%. Importantly, to achieve the same level of sensing quality, we found that the sensor investment can be reduced by over 33\% when considering active bus scheduling. Comprehensive comparative and sensitivity analyses are presented to generate managerial insights and recommendations for practice.
	\end{abstract}


	\section{Introduction}\label{secIntro}
	
	With the rapid development of wireless and low-cost sensors, vehicle-based mobile sensing, also known as drive-by sensing (DS), has become an increasingly important means of surveying spatial-temporal urban environment by leveraging the mobility of public and private transport vehicles \citep{lee2010survey, ma2015opportunities, lane2008urban, liu2005mobility}. 
	DS has been widely adopted in various urban sensing scenarios such as air pollution sensing \citep{gryech2020moreair, ma2008air, song2020deep}, traffic state estimation \citep{du2014effective, li2008performance} and infrastructure health monitoring \citep{eriksson2008pothole, wang2014framework}. The types of host vehicles commonly seen in DS include taxis \citep{o2019quantifying, chen2017trajectory}, buses \citep{cruz2018coverage, kaivonen2020real}, and dedicated vehicle \citep{messier2018mapping}.
	
	The spatial-temporal mobility patterns of the host vehicles have a considerable impact on the sensing quality \citep{anjomshoaa2018city}. For taxis, despite the long operational hours and wide travel range, their spatial survey may be biased, driven by revenue-oriented operations \citep{o2019quantifying}. In contrast, unmanned aerial vehicles or dedicated vehicles have fully controllable routes and schedules \citep{moawad2021real, fan2021towards}, but their large-scale deployment is limited by the high deployment and maintenance costs. As a third option, DS based on buses (or other fixed-route public transport vehicles) not only requires low set-up and maintenance costs, but also enjoys wide spatial and temporal coverage with easy-to-predict vehicle trajectories \citep{wang2018maximizing}.

	Existing studies on bus-based DS have focused on (1) characterizing the spatial-temporal coverage \citep{cruz2018coverage}, and (2) assigning sensors to buses or lines to optimize the sensing quality \citep{kaivonen2020real, wang2018maximizing, james2012sensor}. The latter is usually formulated as a set covering problem \citep{james2012sensor} or a subset selection problem \citep{tonekaboni2020spatio, wang2018maximizing}. These studies are strategic in nature, while the potential of tactical or operational maneuver (e.g. active buses scheduling) has not been explored in the context of bus-based drive-by sensing. This paper addresses this gap by simultaneously optimizing sensor-line assignment, bus dispatch, and single- or multi-line bus relocation. The premise is that, by allowing active bus scheduling within or across lines, the sensors can be better circulated in space and time to achieve superior sensing quality through some centralized decision making procedure, which also guarantees that existing timetables are fulfilled by the fleet, and that additional operational costs associated with active scheduling (e.g. relocation cost) are minimized. 
	
	\begin{figure}[h!]
		\centering
		\includegraphics[width=\textwidth]{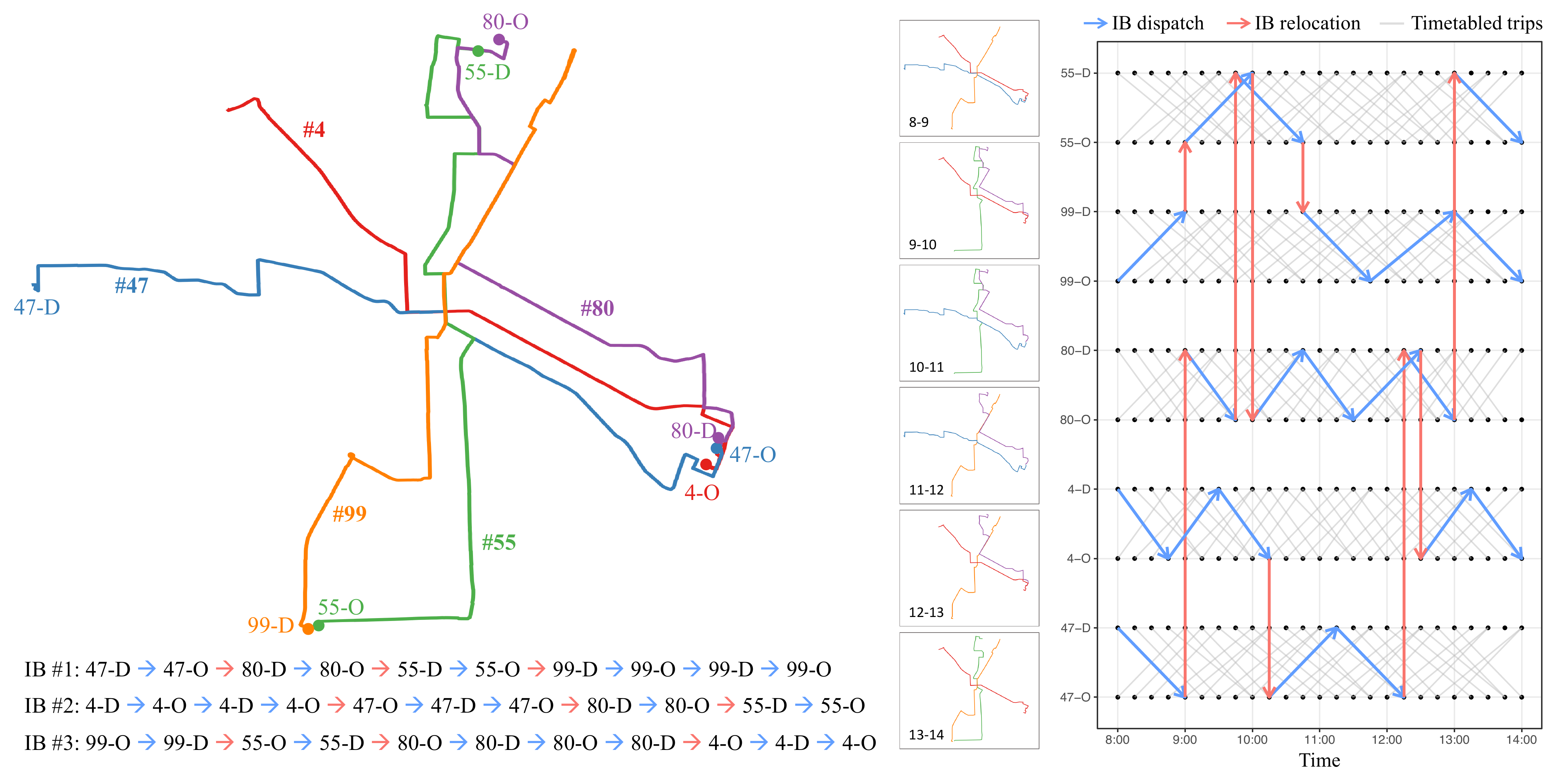}
		\caption{Illustrative example of bus schedules for 3 IBs.}
		\label{figintro}
	\end{figure}

	To gain some intuition into what this work aims to achieve, we illustrate a simple case of multi-line scheduling for instrumented buses (IBs). Figure \ref{figintro} shows a bus transit network with 5 lines. The diagram on the right shows the schedules of 3 IBs during 8:00-14:00. Through inter-line relocation, their spatial coverage is extended to 5 lines instead of 3 (which would be the case in the conventional sensor deployment problem). Moreover, the scheduling procedure ensures sufficient sensor coverage of different locations at different hours (middle of Figure \ref{figintro}). To optimize such a sophisticated mechanism while serving all timetabled trips (using normal buses) and minimizing the bus operational costs is a highly challenging issue.

	This paper addresses both single-line and multi-line scheduling problems for a mixed fleet consisting of normal and instrumented buses. The main contributions and impacts are as follows.
    \begin{itemize}		
		\item We are the first to explore the drive-by sensing potential of buses through active scheduling, advancing relevant literature from strategic sensor deployment to tactical/operational fleet maneuver.  
		
		\item Using time-expanded network, we formulate the multi-line drive-by sensing bus assignment \& scheduling problem as a nonlinear integer program, which subsumes both sensor-line assignment and multi-line bus scheduling problems. Realistic features such as service, wait, pull-out/in, and relocation arcs are included, as well as enforcement of existing timetables, to ensure the operational feasibility and practical relevance of the proposed schemes. 
		
		\item We devise a batch scheduling algorithm to solve this problem, accompanied by theoretic bounds of optimality. A battery of numerical tests show very good solution quality generated by the algorithm (with optimality gap below $3.1\%$), which outperforms Gurobi in terms of both solution quality and computational efficiency. 
		
		\item Based on real-world data in Chengdu, China, a comparative study of bus DS strategies with varying levels of operational intervention is undertaken, together with comprehensive performance evaluation and sensitivity analysis in a wide range of scenarios pertaining to sensor saturation and sensing granularity. These results not only highlight the significant sensing power unlocked by multi-line bus scheduling, but also generate insights regarding the applicability and cost-effectiveness of this approach. 
		
    \end{itemize}

    The proposed framework has the potential to increase the sensing power of bus fleet while saving operational costs; it could also considerably improve the cost-effectiveness of sensor investment. The methodological framework and managerial insights developed in this work is generalizable to a wide range of urban drive-by sensing scenarios (e.g. monitoring of air quality, traffic state, road surface condition, and heat island phenomena).

	The reminder of this paper is organized as follows. Section \ref{secLR} reviews literature related to the topic of this paper. Section \ref{secPD} articulates the model and formulation via the notion of time-expanded network. The solution algorithm is devised in Section \ref{secSA}. Section \ref{secCE} presents some computational  results and analyses. Section \ref{secConclude} provides some concluding remarks and recommendation for practice.

	\section{Related work}\label{secLR}
	
	Work related to this study is divided into three parts: bus-based urban drive-by sensing (DS), vehicle maneuvers in DS, and multi-line bus scheduling.
	
	\subsection{Bus-based drive-by sensing}
	
	Bus-based DS in urban areas is a highly efficient means of ubiquitous sensing with high reliability, low operational costs and wide spatial-temporal coverage \citep{dong2015mosaic, marjovi2015high}. \cite{cruz2018coverage} analyze the mobility patterns of over 5700 buses in Rio de Janeiro, and find that approximately 18\% of the fleet can cover at least 94\% of the streets served by buses. \cite{cruz2020per} further propose a metric for quantifying the spatial-temporal coverage of bus fleet in various application scenarios.
	
	Another line of research focuses on sensor allocation to maximize the sensing quality of bus-based DS. These models can all be characterized as subset selection problems, by identifying a subset of lines or buses to be instrumented \citep{gao2016mosaic, kaivonen2020real}. \cite{james2012sensor} solve the sensor deployment problem by selecting a subset of bus lines based on a chemical reaction optimization algorithm. \cite{ali2017coverage} design a greedy heuristic to install a limited number of sensors on the London bus fleet to maximize the total number of road segments in surface condition monitoring. \cite{wang2018maximizing} design an approximate algorithm to select the subset of buses based on real T-Drive trajectory dataset to maximize the spatial-temporal coverage. \cite{tonekaboni2020spatio}  devise a heuristic algorithm to solve the sensor-bus allocation problem in their analysis of the urban heat island phenomenon, considering the importance of different geographical regions.
	
	In summary, all these studies have considered the bus DS problem from the strategic viewpoint of sensor deployment, which is a non-trivial extension of the classical facility location problems \citep{wang2021emergency, rifki2020impact}:  Instead of geographical locations as candidates, the sensors are mounted onto moving hosts. But the distinction extends far beyond that: Because of the mobility of the hosts, the temporal dimension is explicitly introduced to the decision space, raising its dimension and complexity by a great margin. This paper addresses an even more sophisticated scenario where the hosts (buses) can be actively assigned and scheduled, to boost the sensing power of bus fleets.

	\subsection{Vehicle maneuvers in drive-by sensing}
	
	The majority of DS related research is based on opportunistic sensing \citep{lane2008urban, asprone2021vehicular}, meaning that other than sensor-vehicle pairing, no further maneuvers or intervention is introduced to influence the trajectories of the fleet. In contrast, \cite{chen2020pas} propose a prediction-based actuation system for ride-hailing fleet, where the drivers receive monetary incentives to execute orders that benefit the overall sensing quality. \cite{xu2019ilocus} design a similar incentivizing scheme for taxi drivers, such that the collective sensing profile of the fleet matches a prescribed target distribution. \cite{asprone2021vehicular} consider taxi fleets and realize their sensing goals by generating $\varepsilon$-perturbations of the shortest route between a given origin-destination pair via the enhanced $A^*$ algorithm, and choosing one with the maximum sensing gain. Other than public transport vehicles, dedicated sensing vehicles (DSV) have been considered for their controllability and flexibility. \cite{fan2021towards} consider DSV fleet that are fully controlled by a centralized platform, and optimizes the scheduling policy of DSVs to ensure a fine-grained spatio-temporal sensing coverage.
	
	In summary, literature on active vehicle maneuver in DS is relatively sparse, and unseen for bus fleets. The unique operational characteristics of buses (e.g. fixed routes and timetables) pose new challenges in DS as the coordinated sensing of the fleet is intertwined with sophisticated bus operations. The findings made in this paper could unlock the sensing potential of bus fleets.

	\subsection{Multi-line bus scheduling}

	The multi-line bus scheduling is an extension of the multi-depot vehicle scheduling problem (MDVSP). In the context of bus operations, the MDVSP finds a schedule to serve timetabled trips across multiple lines, by allowing buses to be relocated in an intra- or inter-line fashion. There is a large body of literature investigating the MDVSP. \cite{dauer2021variable} address the MDVSP with heterogeneous fleet and time windows through a time-space network, and propose variable fixing heuristics. \cite{wu2022multi} study the electric bus MDVSP, where buses have limited running range and require long recharging time. The authors adopt a multi-layer time-expanded network, and proposed a branch-and-cut extract solution algorithm. \cite{lu2022combined} adopt a similar multi-layer time-expanded network for combined passenger and parcel transportation, whose layers are specified for each vehicle. \cite{li2021mixed} develop an adaptive time-space-energy network to model the multi-depot vehicle location-routing-scheduling problem with mixed electric and diesel buses, where the location problem refers to the determination of bus refuel/recharge stations. The MDVSP also targets a diverse set of objectives, ranging from maximizing coordinated transfers \citep{ibarra2014integrated} and environmental equities \citep{zhou2020bi}, to minimizing power grid peak load \citep{wu2022multi} and minimizing duration of the longest route \citep{pacheco2013bi}. 
	
	In this work, the sensing objective is realized through the IBs, whose spatial-temporal distributions are jointly determined by bus-to-line assignment and multi-line scheduling. This requires a combination of the sensor deployment problem and the MDVSP in an efficient way, while addressing modeling features pertaining to drive-by sensing and bus operations.

	%\textcolor{red}{The sensor deployment problem (similarly the facility location problem) is well studied in literature. Interested readers could refer to studies of \cite{hammad2017sustainable}, \cite{rifki2020impact}, \cite{yan2016optimal}, and \cite{wang2021emergency}}.

	\section{Preliminaries and model formulation}\label{secPD}
	
	\subsection{Context and assumptions of the problem}
	We consider a multi-line bus transit system with a certain number of mobile sensors installed on a portion of its fleet. It is natural to assume that, once installed, a sensor cannot be removed from the bus. With normal (NBs) and instrumented buses (IBs), the operator assigns buses to lines followed by intra-line and inter-line scheduling to achieve the following goals:
	\begin{itemize}
		\item[(i)] To fulfill existing timetables of all bus lines, thereby serving passenger demands at no cost of the level of service;
		\item[(ii)] To minimize operational costs and fleet size, where costs are incurred during bus pull-out (-in) from (to) the depot, service (to serve timetabled trips), and relocation  processes, and the fleet size should be kept minimum through inter-line coordination;
		\item[(iii)] To optimize spatial-temporal sensing quality, which is defined via a notion of sensing score on discretized space-time domain.
	\end{itemize}

	We consider a multi-line drive-by sensing bus assignment \& scheduling problem, which jointly determines assignment of buses to lines and trip chains (including their relocation processes) of NBs and IBs, to maximize the sensing quality, minimize operational costs, while serving all timetabled trips. We make the following assumptions:
	\begin{itemize}
		\item[(1)]  All IBs and NBs in the system may be relocated to a different terminal upon completion of a service trip, and such a relocation process incurs certain operational cost.  
		
		\item[(2)]  The service and relocation times of buses between terminals are known and deterministic, although they are allowed to change dynamically throughout the day. The buses move along their assigned route at constant speeds.  
		
		\item[(3)] The NBs and IBs are homogeneous in capacity, travel speed, and operational cost.
		
	\end{itemize}

	\subsection{Representation of the bus transit system}\label{subsecTEN}
		Table \ref{tab1} lists key parameters and variables used in the model. 
	
	\begin{longtable}{rl}%{@{\extracolsep{\fill}}rl}
		\caption{Mathematical symbols}	
		\label{tab1} \\
			\hline
			\multicolumn{2}{l}{Indices and sets}     \\
			\hline
			$t$               & The discretized time index for the time-expanded network, $t\in T$
			\\
			$k$               & The discretized time index for evaluating the sensing quality, $k \in T^s$
			\\
			$\Delta_k$       & The temporal sensing granularity, which is the step size of the discrete times $k$
 			\\
			$s,\,e$                    & The start and end time of the planning horizon
			\\
			$i,\,j$                    & The bus terminal index, $i,\,j\in \Theta$
			\\
			$g$                        & The grid index in the target area, $g\in G$
			\\
			$b$                        & The bus index, $b\in B$
			\\
			$\Theta$                    & The set of depots
			\\
			$G$                        & The set of grids, $G=\{g_1,\,g_2,\,\ldots,g_n\}$
			\\
			$B$                        & The set of all buses
			\\
			$B_{\text{IB}}$                        & The set of instrumented buses (IBs)
			\\
			$B_{\text{NB}}$                        & The set of normal buses (NBs)
			\\
			$N$ & The set of timed nodes in the time-expanded network
			\\
			$A$                   & The set of arcs in the time-expanded network
			\\
			$A_d$                   & The set of bus service arcs
			\\
			$A_r$                   & The set of bus relocation arcs
			\\
			$A_w$                   & The set of bus wait arcs
			\\
			$A_p$                   & The set of pull-in/pull-out arcs
			\\
			\hline
			\multicolumn{2}{l}{Parameters}                                                                                                                    \\
			\midrule
			$M$                 & The maximum number of IBs
			\\
			$c_{ij}^{t\bar t}$     & The bus travel cost on arc $\big((i,t),\,(j,\bar t)\big)\in A$ 
			\\
			$\mu_{gk}$      & The spatial-temporal sensing weight for timed grid $(g,k)$
			\\
			$\beta_{ijg}^{t\bar t k}$   & Equals $1$ if $(g,\,k)$ is covered by arc $\big((i,j),\,(j,\bar t)\big)$, and $0$ otherwise                                                                                                                  \\
			$\delta$                   & The weighting factor of objectives 
			\\
			\hline
			\multicolumn{2}{l}{Intermediate decision variables}                                                                                                                 \\
			\hline
			$z_{ijb}^{t\bar t}$       & Equals $1$ if bus $b$ is an IB and traverses arc $\big((i,t),\,(j,\bar t)\big) \in A$, and 0 otherwise \\
			$q_{gk}$                   & The IB sensing times for the timed grid $(g,\,k)$
			\\
			$r_{gk}$              & The effective sensing times for the timed grid $(g,\,k)$                                                                                                                                              \\
			\hline
			\multicolumn{2}{l}{Decision variables}                                                                                                                              \\
			\hline
			$x_{b}$ & Equals $1$ if bus $b$ is an IB, and $0$ otherwise
			\\
			$v_{bi}$ & Equals $1$ if bus $b$ is assigned to depot $i$, and $0$ otherwise
			\\
			$y_{ijb}^{t\overline{t}}$ & Equals $1$ if bus $b$ traversed arc $\big((i, t), (j, \overline{t})\big)\in A$, and $0$ otherwise                                                                                                                                                    \\\hline
	\end{longtable}

	 \begin{figure}[h!]
    		\centering
    		\includegraphics[width=\textwidth]{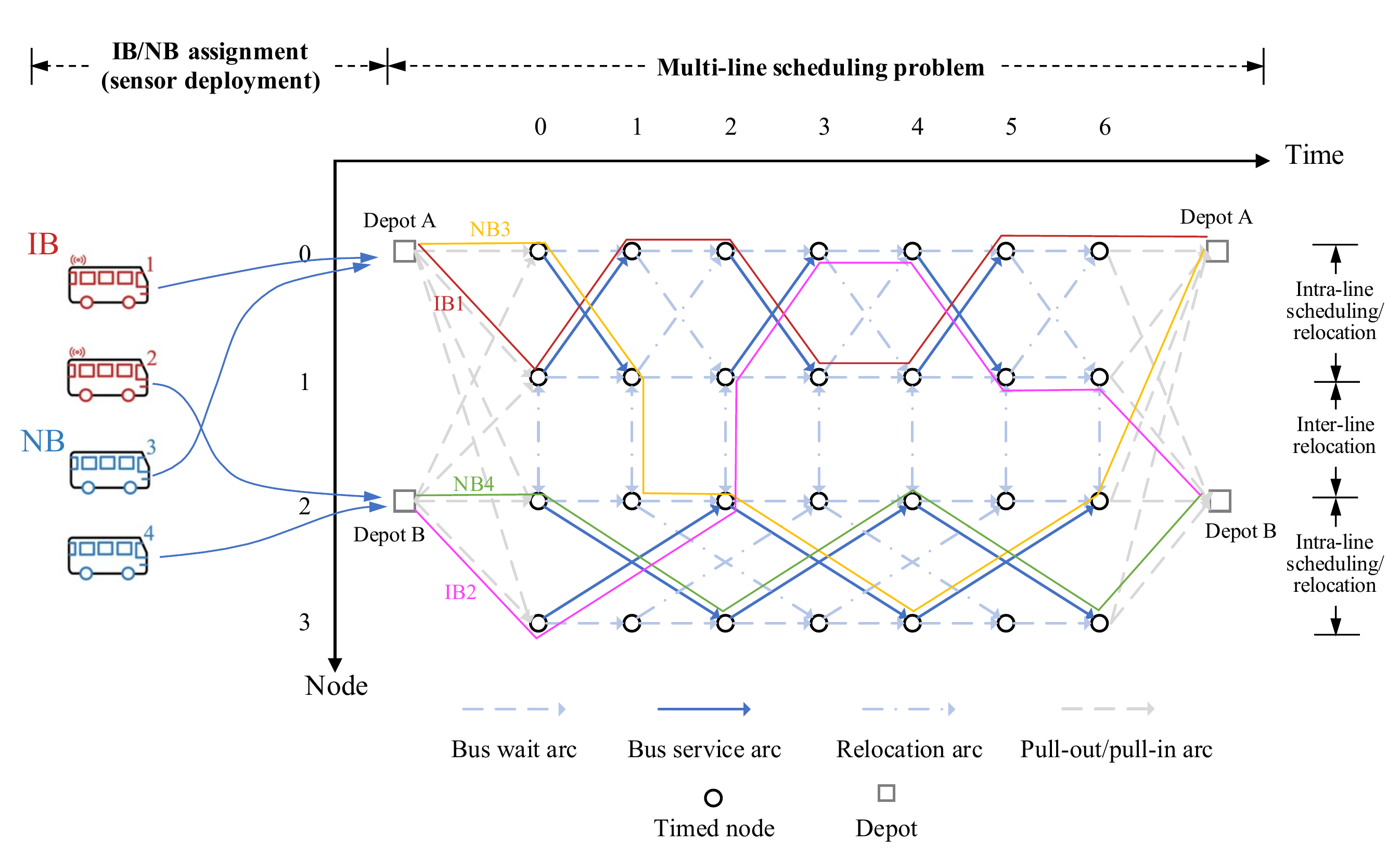}
    		\caption{Illustration of the time-expanded network representation of bus assignment \& scheduling. The colored lines represent a feasible solution of the IB/NB assignment and multi-line scheduling problem, which ensures that all timetabled trips are served, and buses are pulled from/returned to their assigned depots. The solution also has two instances of inter-line bus relocations (NB3 \& IB2).}
    		\label{fig1}
	\end{figure}

	We consider a multi-line bus transit system with IBs and NBs, and study it as a time-expanded network. Fig. \ref{fig1} illustrates such a network with two lines and four terminals, of which only two  (0 and 2) have the functionality of dispatching and holding buses. The network contains timed nodes $(i,\,t)\in N$ representing terminal status across discretized times, and arcs $\big((i,t),\,(j,\bar t)\big)\in A$ denoting possible bus activities. We design four types of arcs for modeling bus-related activities.
	
	{\bf Bus service arcs $A_d\subset A$}: these arcs encode predefined timetabled bus trips. For example, a dispatch arc $\big((i,t),\,(j,\bar t)\big)\in A_d$ is present when there is a scheduled bus service from terminals $i$ to $j$ with start time $t$ and end time $\bar t$. Such service arcs must be covered by an IB or NB. The arc cost is the travel cost for a bus moving from terminals $i$ to  $j$.
	
	{\bf Bus relocation arcs $A_r\subset A$}: these arcs are drawn between terminals with no scheduled bus services.
	IBs traverse these arcs to be allocated to a different line in the interest of drive-by sensing. NBs utilize these arcs to be repositioned to balance the supply in view of IB relocation. The relocation arc cost is also represented by the bus travel cost. As we assume buses must return to their initial depots at the end of the planning horizon, one instance of relocation would require at least a second one to return the bus. Thus, bus relocation costs are to be minimized, among other objectives.

	{\bf Bus wait arcs $A_w\subset A$}: these arcs allow NBs and IBs to wait at terminals for certain time, to facilitate the simultaneous scheduling of NBs and IBs. We impose zero costs on these wait arcs.
	
	{\bf Pull-out/pull-in arcs $A_p\subset A$}: A pull-out arc represents bus movement from the initial depot to its first assignment, connecting the depot node to a zero-timed node (terminal) in the time-expanded network. 
	Pull-out arcs exist between the initial depot to all the multi-line terminals, so that buses are not bound up with its positioned depot but can be dispatched to other lines at the very beginning. The cost of pull-out arc consists of the fixed cost of a bus and the travel cost for a bus moving from depot to its first assignment. A pull-in arc connects the last timed node to the depot, so that buses can be returned to their initial stop at the end of the planning horizon. The cost of pull-in arc is the bus travel cost.
	
	Based on the above time-expanded network, we are able to model and optimize NB and IB operations with bus dispatching, relocation and waiting processes over time.

	\subsection{Definition of drive-by sensing objectives}

	The quantification of sensing quality is essential to the proposed problem. We mesh the target area into $n$ grids indexed as $g\in G=\{g_1,\,g_2,\,\ldots,g_n\}$. 
	The shape and diameter of these grids can be user defined according to the sensing granularity at the data request end. We also introduce discrete times $k\in T^s$ for quantifying sensing quality, which is different from the time index $t$ in the time-expanded network for bus operations. 
	
	The total number of IBs traveling through grid $g$ during time $k$ is
	$$
	q_{gk}=\sum_{b\in B}\sum_{\big((i,t),\,(j, \bar t)\big)\in A}x_b \, y_{ijb}^{t\bar t}\, \beta_{ijg}^{t\bar t k}
	$$
	where $\beta_{ijg}^{t\bar t k}$, $x_b$ and $y_{ijb}^{t \bar t}$ are explained in Table \ref{tab1}. For every grid-time pair $(g,k)$, we assign a spatial-temporal sensing weight of $\mu_{gk}$ to indicate its relative priority, which can be tailored to accommodate  various sensing scenarios and requirements. The weights satisfy $\sum_{g \in G}\sum_{k \in T^s}\mu_{gk}=1$. A straightforward objective would be:
	$$
	\sum_{g\in G}\sum_{k\in T^s}\mu_{gk}q_{gk}
	$$
 
	\noindent However, optimizing the above quantity could lead to undesirable outcome where the coverage is overly concentrated in areas with high weights. To avoid this case, we introduce the diminishing marginal sensing effect by defining the \textit{effective sensing times} $f(q_{gk})$, where $f(\cdot)$ satisfies: 
	\begin{equation}\label{fdef}
	f(0)=0,\qquad f'(q_{gk})>0,\qquad f''(q_{gk})<0
	\end{equation}
 
	\noindent In prose, as the coverage $q_{gk}$ accumulates, the effective sensing times increase as well, but with diminishing marginal gain. This feature is necessary as it discourages excessive coverage of the same grid-time pair $(g,k)$, thereby promoting a more balanced distribution in space and time of the coverage. This leads to the definition of \textit{sensing score} $\Phi$ as a performance indicator of the sensing quality:  
	\begin{equation}\label{STAESTeqn}
		\text{Sensing Score:}~~\Phi = \sum_{g \in G}\sum_{k \in T^s}\mu_{gk}f(q_{gk})
	\end{equation}
	\begin{remark}
	In view of \eqref{fdef}, one choice of the effective sensing times function is $f(q)=q^{0.5}$, which will be used in this paper. Obviously, the form of $f(\cdot)$ depends on the subject of sensing, such as its spatial-temporal distribution and volatility (e.g. air quality or road surface roughness). It also depends on the intended purpose of drive-by sensing (e.g. information collection or emergency response). Another important parameter is the step size $\Delta_k$ of the discrete time $k$: given $f(\cdot)$, larger $\Delta_k$ implies lower requirement on the sensing frequency or data quantity. For this reason, we coin $\Delta_k$ the {\it temporal sensing granularity}. We will run a test on different $\Delta_k$'s later in Section \ref{subsubsecimpact} to understand the performance of the proposed scheme in different sensing scenarios. 
	\end{remark}

	\subsection{Multi-line drive-by sensing bus assignment \& scheduling problem}
	
The problem of interest is analyzed through a nonlinear integer programming model as follows.

	\begin{equation}\label{FP1}
		\min z=\sum_{b\in B}\sum_{\big((i,t),\,(j,\bar t)\big)\in A}c_{ij}^{t\bar t}\,y_{ijb}^{t\bar t} -\delta \cdot \sum_{g\in G}\sum_{k\in T^s}\mu_{gk} f(q_{gk})
	\end{equation}
	\begin{eqnarray}
		\label{FP2}
		\sum_{i\in \Theta}v_{bi}\leq 1 & & \forall b\in B
		\\
		\label{FP3}
		\sum_{\big((i,s),\,(j,\bar t)\big)\in A_p} y_{ijb}^{s\bar t} \leq v_{bi} & & \forall b\in B, i \in \Theta
		\\
		\label{EQ1}
		\sum_{\big((i,s),\,(j,\bar t)\big)\in A_p} y_{ijb}^{s\bar t} -\sum_{\big((j,\bar t), (i,e)\big)\in A_p} y_{jib}^{\bar te}=0 & & \forall b\in B, i \in \Theta
		\\
		\label{FP4}
		\sum_{\big((j,\bar t),\,(i,t)\big)\in A}y_{jib}^{\bar t t}-\sum_{\big(i,t),\,(j,\bar t)\big)\in A}y_{ijb}^{t\bar t}=0\ & &\forall b\in B, (i,\,t)\in N
		\\
		\label{FP5}
		y_{jib}^{\bar t t}+y_{ijb}^{t\bar t}\leq 1 & & \forall b\in B,\big((i,t),\,(j,\bar t)\big)\in A_r
		\\
		\label{FP6}
		\sum_{b\in B}y_{ijb}^{t\bar t}\geq 1& & \forall \big((i,t),\,(j,\bar t)\big)\in A_d
		\\
		\label{FP7}
		\sum_{b\in B}x_b\leq M & & 
		\\
		\label{FP8}
		q_{gk}=\sum_{b\in B}\sum_{\big((i,t),\,(j,\bar t)\big)\in A}x_b\,y_{ijb}^{t\bar t}\,\beta_{ijg}^{t\bar t k} & & \forall g\in G,\,k\in T^s
		\\
		\label{FP9}
		x_b\in\{0,\,1\} & &\forall b\in B
		\\
		\label{FP10}
		v_{bi}\in\{0,\,1\} & &\forall b\in B, i\in \Theta
		\\
		\label{FP11}
		y_{ijb}^{t\bar t}\in\{0,\,1\} & & \forall b\in B,\big((i,t),\,(j,\bar t)\big)\in A
	\end{eqnarray}
	
	The objective function \eqref{FP1} aims to minimize total operational costs of NBs and IBs, while maximizing the sensing score $\Phi$, by introducing a weighting factor $\delta$ to indicate their relative importance. The parameter $\delta$ depends on the decision maker's preference towards both objectives, and is analyzed via a trade-off analysis presented in Section \ref{subsecPD}.
	Constraint \eqref{FP2} ensures that a bus is assigned to at most one depot.  
	Constraint \eqref{FP3} states that buses can only dispatch from their assigned depot. 
	Constraint \eqref{EQ1} guarantees that buses are pulled out and pulled in from the same depot. Such setting ensures that dispatched buses return to the same initial depot at the end of planning horizon, and therefore the balanced fleet management that is usually required by operators. 
	Constraint \eqref{FP4} is the flow balance constraint. Such constraint enables NBs and IBs to find feasible path on the time-expanded network. 
	Constraint \eqref{FP5} eliminates routing sub-tours for same terminal bus relocation across two lines. 
	Constraint \eqref{FP6} guarantees that a timetabled trip is covered by at least one bus, either a NB or an IB.  
	Constraint \eqref{FP7} enforces a maximum IB fleet size of $M$. 
	Constraint \eqref{FP8} calculates the IB sensing times for a grid-time pair $(g,k)$.
	Constraints \eqref{FP9}-\eqref{FP11} defines binary decision variables of the problem.

	\subsection{Model linearization}\label{subsecML}
	
	In this section, we apply linearization techniques to the proposed model. 
	Firstly, we linearize the multiplication $x_b\,y_{ijb}^{t\bar t}$ in \eqref{FP8}, by introducing $z_{ijb}^{t\bar t}=x_b\,y_{ijb}^{t\bar t}$, which is equivalent to the following set of equations:
	\begin{eqnarray}
		\label{FP12}
		z_{ijb}^{t\bar t}\leq x_b & & \forall b\in B, \big((i,t),(j,\bar t)\big)\in A
		\\
		\label{FP13}
		z_{ijb}^{t\bar t}\leq y_{ijb}^{t\bar t} & & \forall b\in B,  \big((i,t),(j,\bar t)\big)\in A
		\\
		\label{FP14}
		z_{ijb}^{t\bar t}\geq x_b+y_{ijb}^{t\bar t} -1 & & \forall b\in B,  \big((i,t),(j,\bar t)\big)\in A
		\\
		\label{FP15}
		z_{ijb}^{t\bar t}\geq 0 & & \forall b\in B, \big((i,t),(j,\bar t)\big)\in A
	\end{eqnarray}
	
	Secondly, we linearize the function $f(q)=q^{0.5}$ in \eqref{FP1} via piecewise affine approximation $\hat f(\cdot)$ (see Fig. \ref{linearization} for an example):
	\begin{equation}\label{hatfdef}
		r_{gk}=\hat f(q_{gk})=m_l q_{gk}+ c_l,\qquad l=1,\ldots,L
	\end{equation}
	
	\noindent Instead of introducing additional binary variables to indicate which line segment is binding, we simply impose the following linear inequalities:	
	\begin{equation}\label{linear2}
		r_{gk} \leq m_l\cdot q_{gk}+c_l \qquad \forall l=1,..,L
	\end{equation}
	\begin{proposition}
		The piecewise affine approximation \eqref{hatfdef}, in case of the objective \eqref{FP1}, is equivalent to the set of linear inequalities \eqref{linear2}.
	\end{proposition}
	\begin{proof}
		The feasibility region described by \eqref{linear2} is the area beneath the piecewise affine approximation \eqref{linear2} (shown as the blue curve in Fig. \ref{linearization}). Since the second term of the objective \eqref{FP1} maximizes the weighted sum of $f(q_{gk})$, which is now approximated by $\hat f(q_{gk})$, this is equivalent to enforcing the largest value of $r_{gk}$ within the aforementioned area, which is precisely $\hat f(q_{gk})$.
	\end{proof}
	
	\noindent As an example, a three-segment approximation of $f(q)=q^{0.5}$ is illustrated in Fig. \ref{linearization}, where the piecewise affine function is given as:
	\begin{equation}\label{example}
		r_{gk} =\hat f(q_{gk})= \left \{\begin{array}{lll}
			q_{gk}                  && 0\leq q_{gk} < 1\\
			0.366\cdot q_{gk}+0.634 && 1\leq q_{gk} < 3\\
			1.732                   && q_{gk} \geq 3
		\end{array}\right.
	\end{equation}
	
	\begin{figure}[h!]
		\centering
		\includegraphics[width=.5\textwidth]{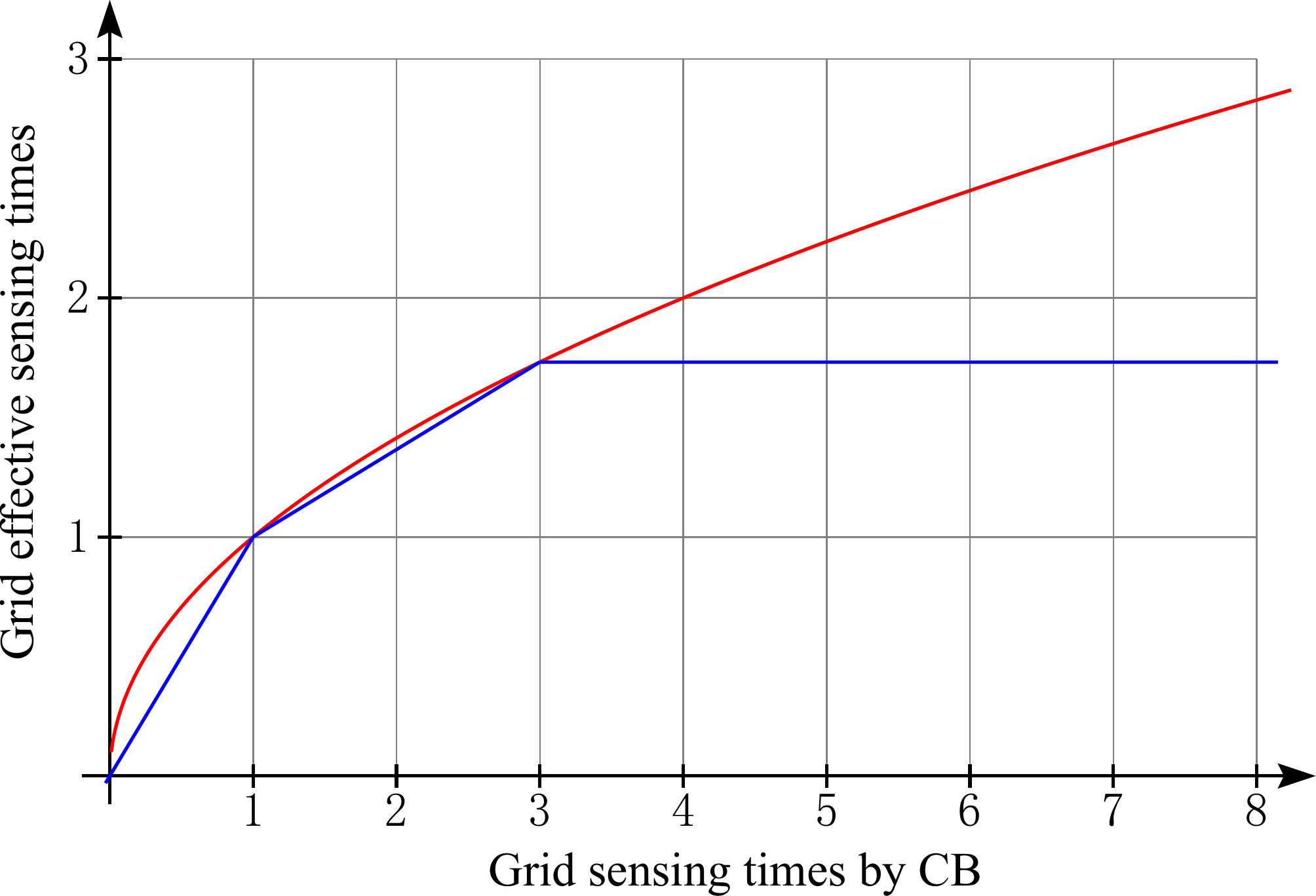}
		\caption{Piecewise affine approximation of the effective sensing time function $f(q)=q^{0.5}$.}
		\label{linearization}
	\end{figure}

	\noindent Finally, the linearized model can be written as:
	\begin{align}
		\label{FP23}
		& \min z=\sum_{b\in B} \sum_{\big((i,t),\,(j,\bar t)\big)\in A}c_{ij}^{t\bar t}\,y_{ijb}^{t\bar t}-\delta \cdot \sum_{g\in G}\sum_{k\in T^s}\mu_{gk}r_{gk}
		\\
		\label{FP24}
		& q_{gk}=\sum_{b\in B}\sum_{\big((i,t),\,(j,\bar t)\big)\in A}z_{ijb}^{t\bar t}\beta_{ijg}^{t\bar t k} \qquad \forall g\in G,\,k\in T^s
		\\
		\label{FP25}
		& \text{Eqns \eqref{FP2}-\eqref{FP7}, \eqref{FP9}-\eqref{FP15}, \eqref{linear2}} 
	\end{align}

	\section{A batch scheduling solution algorithm}\label{secSA}
	\subsection{Algorithm description}\label{subsecAD}
	
	The multi-line drive-by sensing bus assignment \& scheduling problem is an extension of multiple depot vehicle scheduling problem (MDVSP) with heterogeneous fleet, and therefore the model is NP-hard. In order to solve large-scale instances of the problem efficiently, we develop a batch scheduling heuristic algorithm. The main idea is that we optimize schedules of IBs and NBs separately with different objectives, so that the coupling constraints \eqref{FP6}, \eqref{FP8} between IBs and NBs can be relaxed. In this regard, the proposed batch scheduling algorithm could be viewed as an approximation of the original model. The algorithm consists of three main steps. 
	
	\vspace{0.1 in}
	\noindent {\bf Step 1.} Determine the schedule of IB fleet through an IB scheduling sub-model.
	Compared with the original proposed model that optimizes IB and NB scheduling simultaneously, the sub-model focuses on IB scheduling by making the following updates:
	\begin{itemize}
		\item[(i)] We set the objective of the IB scheduling sub-model to maximize trip coverage plus effective sensing and minus relocation travel cost, to encourage IBs to cover as many trips as possible while performing drive-by sensing tasks. 
		
		\item[(ii)] We introduce the parameter $\omega$  to indicate the relative importance of DS gain over trip coverage. Since the value of $\omega$ may impact the two objective values, we will apply a linear search to identify favorable values of  $\omega$ in {\bf Step 3} (a sensitivity analysis of $\omega$ is presented in Appendix \ref{appendixA}).
		
		\item[(iii)] We introduce new binary variables  $h_{ij}^{t\bar t}$ to represent whether timetabled trip $\big((i,t),\,(j,\bar t)\big)\in A_d$ is covered by an IB or not, and add new constraint $\sum_{b\in B_{\text{IB}}}y_{ijb}^{t\bar t}\geq h_{ij}^{t\bar t}$ to guarantee that multiple IB coverages for a trip can only be counted once in the objective function;
		
		\item[(iv)] As only IB fleet is considered, we further simplify the calculation of IB sensing times in the original model as  $q_{gk}=\sum_{b\in B_{\text{IB}}}\sum_{\big((i,t),\,(j,\bar t)\big)\in A}y_{ijb}^{t\bar t}\beta_{ijg}^{t\bar t k}$.
	\end{itemize}

	Let $B_{\text{IB}}$ denote the IB bus fleet. The IB scheduling sub-model is presented as follows.
	\begin{equation}\label{SA27}
		\max z_{\text{IB}}=\sum_{\big((i,t),\,(j,\bar t)\big)\in A_d} c_{ij}^{t\bar t}\, h_{ij}^{t\bar t}+\omega\cdot \delta\cdot\sum_{g\in G}\sum_{k\in T^s}\mu_{gk}\,r_{gk} -\sum_{b\in B_{\text{IB}}}\sum_{\big((i,t),\,(j,\bar t)\big)\in A_r}c_{ij}^{t\bar t}y_{ijb}^{t\bar t}
	\end{equation}
	\begin{eqnarray}
		\label{SA28}
		\sum_{i\in \Theta}v_{bi}\leq 1 & & \forall b\in B_{\text{IB}}
		\\
		\label{SA29}
		\sum_{((i,s),(j,\bar t)) \in A_p}y_{ijb}^{s\bar t}\leq v_{bi} & & \forall b\in B_{\text{IB}},i\in \Theta
		\\
		\sum_{\big((i,s),\,(j,\bar t)\big)\in A_p}y_{ijb}^{s\bar t} - \sum_{\big((j,\bar t),(i,e)\big)\in A_p}y_{jib}^{\bar te}=0 & & \forall b\in B_{\text{IB}},i\in \Theta
		\\
		\label{SA30}
		\sum_{\big((j,\bar t),\,(i,t)\big)\in A}y_{jib}^{\bar t t} - \sum_{\big((i,t),\,(j,\bar t)\big)\in A}y_{ijb}^{t\bar t}=0 & & \forall b\in B_{\text{IB}},(i,t)\in N
		\\
		\label{SA31}
		y_{jib}^{\bar t t}+y_{ijb}^{t \bar t}\leq 1 & & \forall b\in B_{\text{IB}},\big((i,t),\,(j,\bar t)\big)\in A_r
		\\
		\label{SA32}
		\sum_{b\in B_{\text{IB}}}y_{ijb}^{t\bar t}\geq h_{ij}^{t\bar t} & & \forall \big((i,t),\,(j,\bar t)\big)\in A_d
		\\
		\label{SA33}
		q_{gk}=\sum_{b\in B_{\text{IB}}}\sum_{\big((i,t),\,(j,\bar t)\big)\in A}y_{ijb}^{t\bar t}\beta_{ijg}^{t \bar t k} & &\forall g\in G,\,k\in T^s
		\\
		\label{SA34}
		r_{gk} \leq m_l\cdot q_{gk}+c_l & & \forall g\in G,\, k\in T^s, l=1,..,L
		\\
		\label{SA41}
		v_{bi}\in\{0,\,1\} & & \forall b\in B_{\text{IB}},\,i\in \Theta
		\\
		\label{SA42}
		y_{ijb}^{t\bar t}\in\{0,\,1\} &  & \forall b\in B_{\text{IB}},\,\big((i,t),\,(j,\bar t)\big)\in A
		\\
		\label{SA43}
		h_{ij}^{t\bar t}\in\{0,\,1\} & & \forall \big((i,t),\,(j,\bar t)\big)\in A
	\end{eqnarray}
	
	\vspace{0.1 in}
	
	\noindent {\bf Step 2.} Determine the schedule of NBs given the scheduling result of IBs, which is much simpler as its objective is simply to serve unfulfilled timetabled trips at the lowest operation costs. Based on the results of the IB scheduling sub-model, the trip set covered by IBs are obtained as  $A^{\text{IB}}_d=\big\{\big((i,t),\,(j,\bar t)\big)\in A_d \big\vert h_{ij}^{t\bar t}=1\big\}$, where  $A_d$ is the full timetabled trip set. 
	Then, the uncovered timetabled trip set that requires NBs to serve is  $A_{d}^{\text{NB}}=A_d\setminus A^{\text{IB}}_d$. 
	The NB scheduling sub-model is developed as follows.
	
	\begin{equation}\label{SA44}
		\min z_{\text{NB}}=\sum_{b\in B_{\text{NB}}}\sum_{\big((i,t),\,(j,\bar t)\big)\in A}c_{ij}^{t\bar t}\,y_{ijb}^{t\bar t}
	\end{equation}
	\begin{eqnarray}
		\label{SA45}
		\sum_{i\in \Theta}v_{bi}\leq 1 & & \forall b\in B_{\text{NB}}
		\\
		\label{SA46}
		\sum_{((i,s),(j,\bar t))\in A_p}y_{ijb}^{s\bar t}\leq v_{bi} & & \forall b\in B_{\text{NB}},i\in \Theta
		\\
		\label{SA47}
		\sum_{\big((i,s),\,(j,\bar t)\big)\in A_p}y_{ijb}^{s\bar t} - \sum_{\big((j,\bar t),(i,e)\big)\in A_p}y_{jib}^{\bar te}=0 & & \forall b\in B_{\text{NB}},i\in \Theta
		\\
		\sum_{\big((j,\bar t),\,(i,t)\big)\in A}y_{jib}^{\bar t t} - \sum_{\big((i,t),\,(j,\bar t)\big)\in A}y_{ijb}^{t\bar t}=0 & & \forall b\in B_{\text{NB}},(i,t)\in N
		\\
		\label{SA48}
		y_{jib}^{\bar t t}+y_{ijb}^{t\bar t}\leq 1 & & \forall b\in B_{\text{NB}},\,\big((i,t),\,(j,\bar t)\big)\in A_r
		\\
		\label{SA49}
		\sum_{b\in B_{\text{NB}}}y_{ijb}^{t\bar t}\geq 1 & & \forall \big((i,t),\,(j,\bar t)\big)\in A_d^{\text{NB}}
		\\
		\label{SA51}
		y_{ijb}^{t\bar t}\in \{0,\,1\} & & \forall b\in B_{\text{NB}},\big((i,t),\,(j,\bar t)\big)\in A
	\end{eqnarray}
	
	\noindent {\bf Step 3.} Perform linear search of $\omega$ for better trade-off between drive-by sensing and trip coverage, and thereby improving the overall objective. Suppose  $\omega\in[0,\,\bar{\omega}]$, for parameter values  $\omega=0,\,\Delta,\,2\Delta,\,\ldots,\,\bar{\omega}$ with increment $\Delta$, we respectively apply the IB scheduling sub-model and the NB scheduling sub-model to obtain corresponding bus schedules and objective values. The value of $\omega$ with the best objective is selected.
	
	The batch scheduling heuristic is described in Algorithm \ref{alg1}.

	\begin{algorithm}[h!]
		\begin{tabbing}
			\hspace{0.01 in}\=  \hspace{0.2 in}\= \kill % set up two tab positions
			\>  {\bf Step 0: Initialization} \>
			\\
			\>  \>  (1) Set time-dependent grid drive-by sensing weights;
			\\
			\>  \>  (2) Calculate arc costs according to bus travel time;
			\\
			\>   \> (3) Construct the time-expanded network. Classify arcs into sets of bus service arcs, 
			\\
			\> \> relocation arcs, wait arcs, and pull-out/pull-in arcs.
			\\
			\>  {\bf Step 1: Determine IB schedules}   \>
			\\
			\> \> (1) Construct the IB scheduling sub-model by setting parameter $\omega$. The sub-model is
			\\
			\> \> always feasible, as no hard constraints are put on IB scheduling;
			\\
			\> \> (2) Solve the sub-model using Gurobi;
			\\
			\> \> (3) Obtain schedules of each IB and calculate objective counterparts in the original
			\\
			\> \> model;
			\\
			\> \> (4) Obtain covered trip set by IBs as $A_d^{\text{IB}}$.
			\\
			\> {\bf Step 2: Determine NB schedules}  \>
			\\
			\> \> (1) Construct the NB scheduling sub-problem based on the uncovered trip set
			\\
			\> \> $A_d^{\text{NB}}=A_d\setminus A_d^{\text{IB}}$;
			\\
			\> \> (2) Solve the sub-problem using Gurobi;
			\\
			\> \> (3) If the sub-model is solved to optimality, obtain schedules of each NB and calculate
			\\
			\> \> objective counterparts in the original model;
			\\
			\> \> (4) If the sub-problem is infeasible, set the objective value as infinity.
			\\
			\> {\bf Step 3: Perform parameter linear search for objective reduction} \>
			\\
			\> \> (1) Run {\bf Step 1} and {\bf Step 2} sequentially and respectively for parameter values $\omega=\Delta,$
			\\
			\> \> $2\Delta,\,\ldots,\, \bar{\omega}$. Obtain the corresponding objective values;
			\\
			\> \> (2) Determine the best parameter and objective value. The corresponding IB and NB
			\\
			\> \> schedules are deemed the solutions to the  problem.
		\end{tabbing}
		\caption{(The batch scheduling algorithm)}
		\label{alg1}
	\end{algorithm}

	\subsection{Theoretical bounds on optimality}
	To derive theoretic bounds on the solution quality of the proposed batch scheduling algorithm, we construct the following three sub-problems, which are mathematically elaborated in Table \ref{subproblems}.

\begin{itemize}
\item The first sub-problem (SP1) represents normal bus operation without any DS consideration, by minimizing bus operational costs. The optimal solution, which is also the minimum bus operation costs required for serving multi-line timetables, is denoted $C_{\text{NB}}^*$. 

	\item The second sub-problem (SP2) represents an extreme case in the batch scheduling algorithm where the weight $\omega=0$ in \eqref{SA27}. The sub-problem first maximizes trip coverage minus operational cost of IB fleet and then minimizes operational cost of NBs. The total bus operational cost is denoted $C_{\text{BS}}^*$; the DS objective value associated with IB fleet is denoted $R_{\text{BS}}^*$. 
	
	\item The third sub-problem (SP3) represents another extreme case of $\omega \to+\infty$. The sub-model first maximizes DS quality using IB fleet and then minimizes operational cost of NBs. The resulting DS quality and total bus operational costs are denoted $R_{\text{DS}}^*$ and $C_{\text{DS}}^*$, respectively. 
\end{itemize}

\begin{table}[h!]
		\centering
		\caption{Sub-problems for optimality analysis.}
		\begin{tabular}{|c|c|c|c|}
			\hline
			\multirow{3}{*}{Sub-problem} & \multirow{3}{*}{Model detail}  & \multicolumn{2}{c|}{Optimal solution}
			\\\cline{3-4}
			& & Operational  & DS 
			\\
			& & cost & quality
			\\
			\hline
			SP1 &  \makecell[l]{ $\min z=\sum \limits_{b\in B}\sum\limits_{((i,t),\,(j,\bar t))\in A}c_{ij}^{t\bar t}\,y_{ijb}^{t\bar t}$ 
			\\ s.t. Eqns \eqref{FP2}-\eqref{FP6}, \eqref{FP10}-\eqref{FP11}}& $C_{\text{NB}}^*$ & - \\
			\hline
			SP2 & \makecell[l]{ $\max z_{\text{IB}}=\sum\limits_{((i,t),\,(j,\bar t))\in A_d} c_{ij}^{t\bar t}\, h_{ij}^{t\bar t}$ \\ \hspace{0.7 in}$-\sum\limits_{b\in B_{\text{IB}}}\sum\limits_{((i,t),\,(j,\bar t))\in A_r}c_{ij}^{t\bar t}y_{ijb}^{t\bar t}$ 
			\\ s.t. Eqns \eqref{SA28}-\eqref{SA32}, \eqref{SA41}-\eqref{SA43} 
			\\
			$\min z_{\text{NB}}=\sum\limits_{b\in B_{\text{NB}}}\sum\limits_{((i,t),\,(j,\bar t))\in A}c_{ij}^{t\bar t}\,y_{ijb}^{t\bar t}$ \\s.t. Eqns \eqref{SA45}-\eqref{SA51} }& $C_{\text{BS}}^*$ & $R_{\text{BS}}^*$ \\
			\hline
			SP3 & \makecell[l]{$\max z_{\text{IB}}= \sum\limits_{g\in G}\sum\limits_{k\in T^s}\mu_{gk} r_{gk}$ 
			\\ s.t.  Eqns \eqref{SA28}-\eqref{SA31}, \eqref{SA33}-\eqref{SA42} \\ $\min z_{\text{NB}}=\sum\limits_{b\in B_{\text{NB}}}\sum\limits_{((i,t),\,(j,\bar t))\in A}c_{ij}^{t\bar t}\,y_{ijb}^{t\bar t}$ \\s.t. Eqns \eqref{SA45}-\eqref{SA51}}& $C_{\text{DS}}^*$ & $R_{\text{DS}}^*$ 			
			\\\hline
		\end{tabular}
		\label{subproblems} 
	\end{table}

	Note that these sub-problems could be solved to optimality more easily than the original model, since the DS and operational objectives are decoupled, which reduces the complexity arising from the coordination of IBs and NBs. SP1-SP3 are used in Proposition \ref{proposition1} to derive lower and upper bounds of the batch scheduling algorithm, as well as its worst-case performance guarantee.

\begin{proposition}\label{proposition1}
Let the objective value of the exact solution of \eqref{FP23}-\eqref{FP25} be $z^*$, and the objective value of the solution from the batch scheduling algorithm be $z$. The following holds:
\begin{itemize}
			\item {\bf Lower and upper bounds:}  
\begin{equation}\label{LUB}
C_{\text{NB}}^*-\delta\cdot R_{DS}^* \leq z^* \leq z \leq \max\big\{C_{DS}^*-\delta\cdot R_{DS}^*, C_{BS}^*-\delta\cdot R_{BS}^*\big\}
\end{equation}
\item {\bf Gap estimate}. The objective gap of the proposed batch scheduling algorithm satisfies 
\begin{equation}\label{gapB}
\frac{z-z^*}{z^*} \leq \frac{\max\{C_{DS}^*-\delta\cdot R_{DS}^*, C_{BS}^*-\delta\cdot R_{BS}^*\}}{C_{NB}^*-\delta\cdot R_{DS}^*}-1
\end{equation}
		\end{itemize}
	\end{proposition}
	
	\begin{proof}
		Let $C^*$ and $R^*$ respectively be the bus operational costs and DS sensing objective associated with the exact solution $z^*$. Then, $C^*\geq C_{\text{NB}}^*$ because $C_{\text{NB}}^*$ is the minimum bus operation costs required for serving multi-line timetables. It is also true that $R^*\leq R_{\text{DS}}^*$ as $R_{\text{DS}}^*$ is the maximum sensing power of IB fleet. Therefore, we obtain the lower bound for $z$ as:
		\begin{equation}
			z\geq z^*=C^*-\delta\cdot R^*\geq C_{\text{NB}}^*-\delta\cdot R_{\text{DS}}^*
		\end{equation}
		
		\noindent The upper bound of $z$ is obtained in either one of two extreme cases; i.e. IBs are scheduled to save operation costs only (SP2), or to maximize DS gain only (SP3). The reason is that the proposed heuristic generates IB and NB schedules by trading operational costs off with DS gain, its objective should be no worse than that of the two extreme cases. Accordingly, the upper bound can be written as:
		\begin{equation}
			z\leq \max\big\{C_{\text{DS}}^*-\delta\cdot R_{\text{DS}}^*, C_{\text{BS}}^*-\delta\cdot R_{\text{BS}}^*\big\}
		\end{equation}
		
		\noindent Finally, the worst-case relative gap of the proposed heuristic can be estimated as:
		\begin{align}
			\frac{z-z^*}{z^*} &= \frac{z}{z^*}-1\\
			&\leq \frac{z}{C_{\text{NB}}^{*}-\delta\cdot R_{\text{DS}}^{*}}-1\\
			&\leq \frac{\max\big\{C_{\text{DS}}^*-\delta\cdot R_{\text{DS}}^*, C_{\text{BS}}^*-\delta\cdot R_{\text{BS}}^*\big\}}{C_{\text{NB}}^{*}-\delta\cdot R_{\text{DS}}^{*}}-1
		\end{align}
	\end{proof}

	\begin{corollary}
		From Proposition \ref{proposition1} we derive that $\lim_{\delta \to \infty}\frac{z-z^*}{z^*}=0$, suggesting that the proposed algorithm approaches the exact solution when the DS objective dominates the operational costs, i.e. when $\delta \to +\infty$. 
	\end{corollary}

	\section{Computational experiments}\label{secCE}
	This section conducts a series of experiments to evaluate the performance of the batch scheduling algorithm and assess the impact of the proposed bus scheduling method on sensing and operation objectives. 
	Section \ref{subsecDD} describes data and exogenous parameters, including bus timetables, travel times, corresponding costs, and grid sensing weights. 
	Section \ref{subSecBM} details the three compared benchmark methods.
	Section \ref{subsecAP} analyzes algorithm performances. 
	Section \ref{subsecCR} provides optimization results and discussions on: optimal IB and NB schedules, impact of active IB scheduling, impact of the temporal sensing granularity, and trade-off between sensing quality and cost. 
	All experiments are conducted on a Microsoft Windows 10 platform with Intel(R) Core(TM) i7-10700F CPU @ 2.90GHz and 16.0 GB RAM, using Python 3.9.12 and Gurobi 9.5.1. 
	
	\subsection{Description of data and exogenous parameters}\label{subsecDD}

	The case study focuses on the central area of Chengdu, China, which has diverse land use types, including central business districts, residential area, industrial parks and transportation hubs. The area is meshed into 1km-by-1km grids, and the grid sensing weights are calculated based on the land use types, as shown in Fig. \ref{figLines}.
	
	\begin{figure}[h!]
		\centering
		\includegraphics[width=0.7\textwidth]{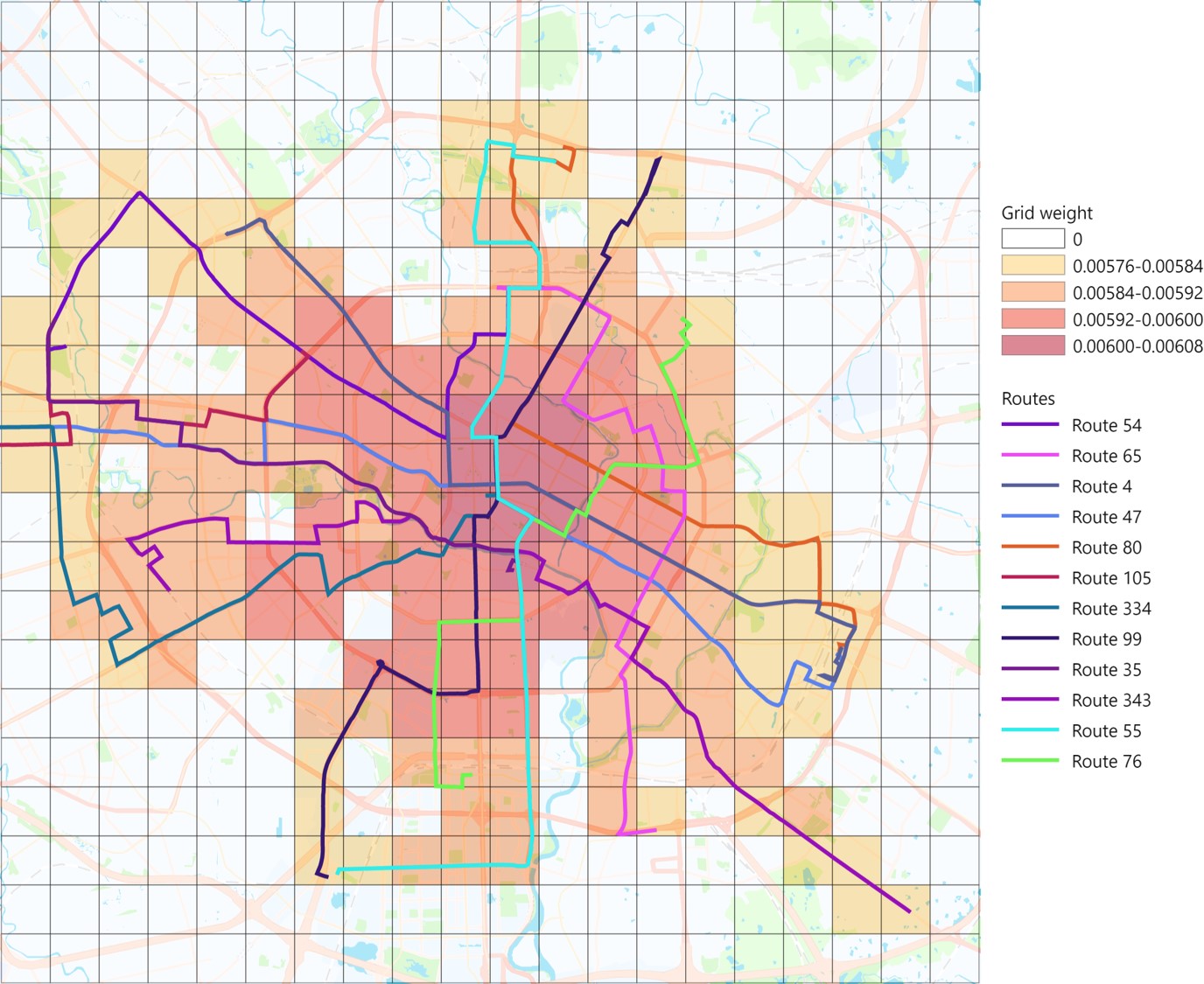}
		\caption{The study area with 12 bus lines and a mesh structure of 1 km $\times$ 1 km}
		\label{figLines}
	\end{figure}

	Twelve bus lines have been selected for the drive-by sensing experiment, which are shown in Fig. \ref{figLines}. The bus timetables are provided by local operator, and bus travel times of these lines are obtained by analyzing bus GPS trajectories during 2021.10.01-2021.10.31. These data are illustrated in a time-expanded manner in Fig. \ref{figTimetable}. It is observed that:  
	\begin{enumerate}
		\item  The timetables are regular but with different dispatch frequencies. For example, the dispatch headway of Line 105 is 30 min, while other lines has a headway of 15 min; 
		\item Bus travel times are time-dependent and varying. For example, bus travel times of Line 55 vary from 60 min to 75 min. 
	\end{enumerate}

	Arc costs mainly include bus fixed costs and travel costs. The fixed bus cost is estimated as 856 RMB for a single day of operation based on the BYD B12 model with a price of 2.5 million RMB and a service life of 8 years. Bus travel cost is estimated as 1.4 RMB per travel minute according to \cite{li2014transit}. The bus relocation cost consists of fixed and variable components: The fixed cost is 20 RMB, while the variable cost is calculated based on the unit travel cost of 1.4 RMB/min. The bus relocation times between two terminals are obtained from the point-to-point travel times provided by Amap.com, a digital map service provider.

	Finally, we adopt the piecewise affine effective sensing time function $\hat f(q_{gk})$ shown in Eqn \eqref{example} and Fig. \ref{linearization}. In all calculations presented in Sections \ref{subsecAP}-\ref{subsecCR}, we set the weighting parameter $\delta=4000$ in Eqn \eqref{FP1}, while explaining such a choice in Section \ref{subsecPD}.
	
	\begin{figure}[h!]   
		\centering
		\includegraphics[width=0.5\textwidth]{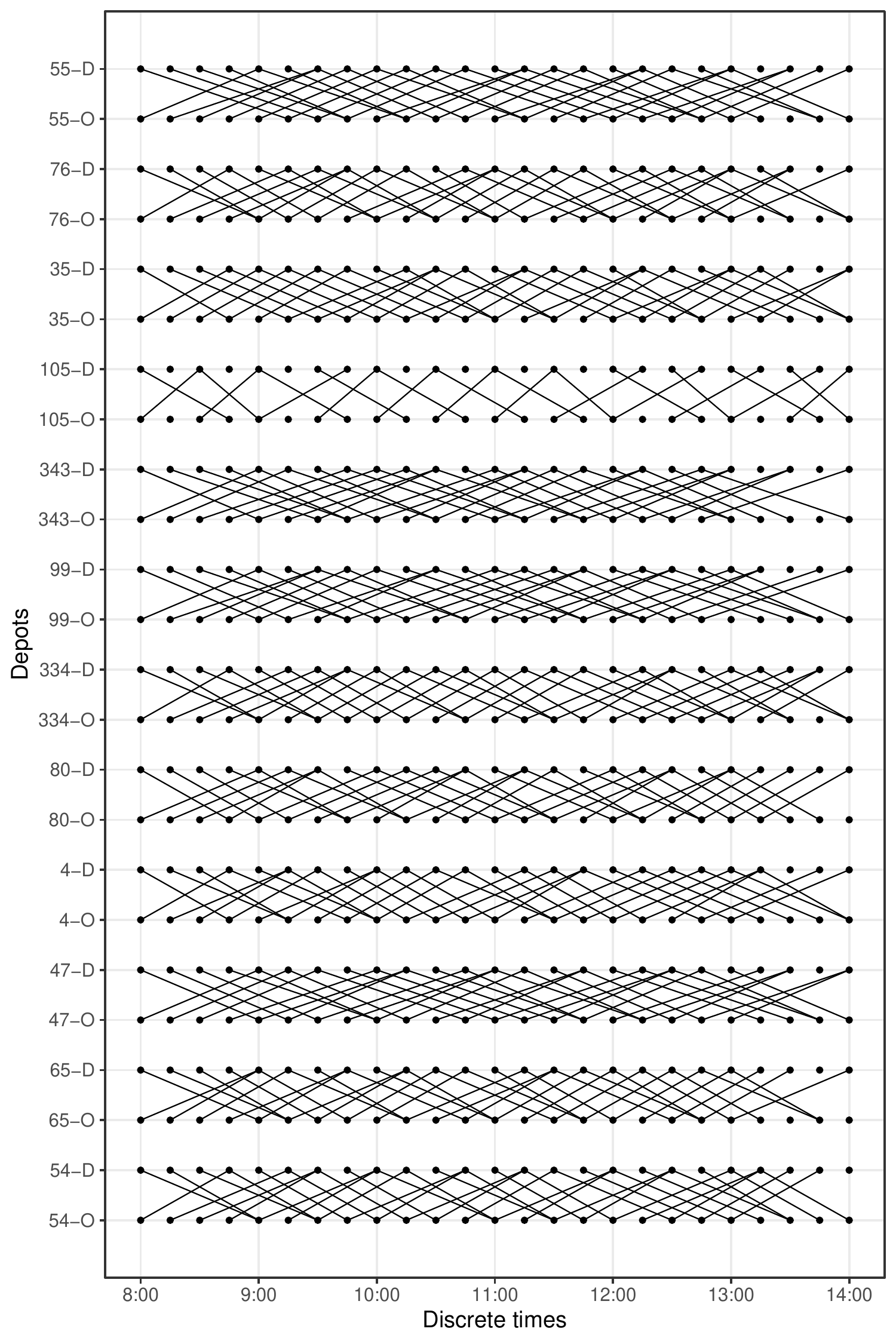}
		\caption{Bus timetables considered in the case study. The discrete time step is 15 min.}
		\label{figTimetable}
	\end{figure}

	\subsection{Different bus drive-by sensing approaches}\label{subSecBM}
	
	We consider the following bus drive-by sensing scenarios to demonstrate the effectiveness of the proposed single-line and multi-line bus scheduling methods. 
	
	\begin{itemize}
		\item[(1)] \textbf{Optimal sensor deployments without active bus scheduling (M1)}: Following \cite{james2012sensor}, this method aims to allocate sensors to bus lines with the objective of maximizing grid coverage percentage. The problem is formulated as a set covering problem as follows:
		
		\begin{equation}
			\max \, z= \sum_{g \in G} u_g
		\end{equation}
		\begin{eqnarray}
			\sum_{r\in R} \upsilon _{sr} \leq 1 && \forall s \in S
			\\
			\sum_{s\in S}\sum_{r\in R} \upsilon _{sr}\phi_{rg} \geq u_g & & \forall g\in G
			\\
			\upsilon _{sr} \in\{0,1\} && \forall s \in S, r\in R
			\\
			u_g \in \{0,1\} && g \in G
		\end{eqnarray}
		
		where the decision variable $\upsilon _{sr}\in \{0,1\}$ indicates whether sensor $s$ is installed to line $r$, and  $u_g\in \{0,1\}$ indicates whether grid $g$ is covered by a line. The binary parameter $\phi_{rg}$ indicates whether line $r$ traverses grid $g$. We note that this formulation solves the sensor-line, rather than sensor-bus, assignment, which means the actual time-varying grid coverage is dependent on subsequent bus scheduling decisions. The performance of this policy is obtained by randomizing the selection of buses for dispatch, without knowledge of sensor instrumentation.

		\item[(2)] \textbf{Optimal sensor deployments with active single-line scheduling (M2)}:  this policy jointly optimize sensor deployment and single-line IB and NB scheduling. 
		The underlying assumption is that buses cannot be relocated to a different line. 
		
		\item[(3)] \textbf{Optimal sensor deployments with active multi-line scheduling (M3)}: this policy jointly optimizes sensor deployment and multi-line IB and NB scheduling. The key difference from M2 is that it allows buses to be relocated to different lines in the network. 
	\end{itemize}
	
	 The three methods represent different states of practice pertaining to bus-based DS. M1 concerns with sensor deployment to a bus fleet but without any operational intervention. This approach is strategic in nature and has been studied in the literature as a set covering problem \citep{james2012sensor} or subset selection problem \citep{tonekaboni2020spatio}. The drawback, however, is  that (i) the scheduling of instrumented buses is not optimized to deliver satisfactory sensing outcome, and (ii) covering an entire city could be very costly since the number of lines under consideration could be huge. M2 and M3 are proposed by this work to boost the sensing efficacy of bus DS through active bus scheduling. In particular, M2 remedies situation (i) above by simultaneously scheduling IBs and NBs within their designated lines to optimize the temporal coordination of the sensors. M3 adds another level of flexibility to bus operations by allowing an IB to be relocated to other lines, thereby expanding the spatial coverage of an IB fleet. This is a promising way to address challenge (ii) above. In summary, the three approaches represent increasing degrees of flexibility in sensor distribution, both temporally and spatially. A comparative study of these methods, as we will undertake in the following sections, will quantify their impact on urban crowdsensing and offer managerial insights for the adoption of bus-based DS.

	\subsection{Algorithm performance}\label{subsecAP}
	
	We assess the proposed batch scheduling algorithm with varying problem sizes. Specifically, we generate 10 instances by varying the number of lines (6, 9, 12, 16 and 20) and the planning horizon (6 and 8 hours). The transit network with 20 lines is shown in Appendix \ref{appendixB}. The IB fleet size is set as 5, while the NB fleet size is determined by the optimization procedure. The maximum computational time is set as one hour for Gurobi. The parameter MipGap in Gurobi is set as 0.01 for all the models. Parameter searching space of $\omega$ is set as $\{0.5, 1.0, 1.5\}$, where a more thorough sensitive analysis of $\omega$ is presented in Appendix \ref{appendixA}. 
	
		\begin{table}[h!]
		\centering
		\caption{Bounds and worse-case performance of the proposed algorithm. Obj and Gap are the actual objective value and gap (to LB) of the proposed algorithm.}
		\begin{threeparttable}	
			\begin{tabular}{ccccccc}
				\hline
				\# Lines & \# Hours & LB & UB & Worst-case gap & Obj & Gap
				\\
				\hline
				\multirow{2}{*}{6} & 6 & 58041.6 & 64423.0 & 11.0\% & 58852.5 & 1.4\%
				\\
				& 8 & 65916.3 & 72360.0 & 9.8\%  & 66815.1 & 1.4\%
				\\\hline
				\multirow{2}{*}{9} & 6 & 86837.5 & 93788.0 & 8.0\% & 87774.2 & 1.1\%
				\\
				& 8 & 97687.5 & 104658.0 & 7.1\%  & 98781.3 & 1.1\% 
				\\\hline
				\multirow{2}{*}{12} & 6 & 118730.5 & 127621.0 & 7.5\%  & 119636.9 & 0.8\% 
				\\
				& 8 & 133650.7 & 145390.0 & 8.8\%  & 134749.7 & 0.8\%
				\\\hline
				\multirow{2}{*}{16} & 6 & 140259.4 & 153423.0 & 9.4\%  & 142379.2 & 1.5\%
				\\
				& 8 & 159850.6 & 174195.6 & 9.0\%  & 164933.3 & 3.2\%
				\\\hline
				\multirow{2}{*}{20} & 6 & 165104.1 & 177006.0 & 7.2\%  & 167627.1 & 1.5\%
				\\
				& 8 & 186967.3 & 202354.7 & 8.2\%  & 193904.4 & 3.7\%
				\\\hline
			\end{tabular}
		\end{threeparttable}	
		\label{tabAGAP} 			
	\end{table}
	
	Table \ref{tabAGAP} shows the theoretical lower (LB) and upper bounds (UB), as well as the worse-case gap, of the proposed algorithm according to Proposition \ref{proposition1}. The parameters used in the formulae \eqref{LUB} and \eqref{gapB} are obtained via Gurobi with 1\% optimality gap. It is remarkable that the actual objective values by the proposed algorithm is very close to the theoretical lower bounds, with all gaps below $3.7\%$, which not only shows a very good solution quality, but also implies that the lower-bound estimate in Proposition \ref{proposition1} is quite sharp.

	Table \ref{tabAP} evaluates the computational performance in terms of optimality gap and computational (CPU) time. The lower bound (LB) provided by the commercial solver Gurobi is used as the benchmark. Firstly, solutions generated by Gurobi have good quality for small instances (12  lines or less), with $\text{Gap}_1$ between 0.29\% and 0.65\%, but deteriorate for the 16- and 20-line cases, with $\text{Gap}_1$ over 12\%. In fact, the time limit (3600s) was reached before Gurobi can converge. On the other hand, the batch scheduling algorithm consistently outperforms Gurobi across all instances, especially for large-scale cases (16 and 20 lines) with at least 78.7\% gap reduction (16 lines, 8 hours), meanwhile saving substantial computational time (by up to 68.8\%).

	\begin{table}[h!]
		\centering
		\caption{Algorithm performance. }
		\begin{threeparttable}	
			\begin{tabular}{ccllll|lll}
				\hline
				\multirow{2}{*}{\# Lines}  &  \multirow{2}{*}{\# Hours} & \multicolumn{4}{c|}{Gurobi} & \multicolumn{3}{c}{Batch scheduling algorithm}
				\\\cline{3-9}
				&  & LB & Obj$_1$ & Gap$_1$ & CPU & Obj$_2$ & Gap$_2$ & CPU
				\\
				\hline
				\multirow{2}{*}{6} & 6 & 58713.0 & 59048.8 & 0.57\% & 17 s & 58852.5 & 0.24\% & 19 s
				\\
				& 8 & 66618.4 & 66810.2 & 0.29\% & 61 s & 66815.1 & 0.30\% & 63 s
				\\\hline
				\multirow{2}{*}{9} & 6 & 87460.0 & 87931.5 & 0.54\% & 108 s & 87774.2 & 0.36\% & 81 s
				\\
				& 8 & 98305.9 & 98953.0 & 0.65\% & 115 s & 98781.3 & 0.48\% & 259 s
				\\\hline
				\multirow{2}{*}{12} & 6 & 119044.7 & 119637.9 & 0.50\% & 1462 s & 119636.9 & 0.50\% & 794 s
				\\
				& 8 & 134159.7 & 134978.4 & 0.61\% & 3600 s & 134749.7 & 0.44\% & 1485 s
				\\\hline
				\multirow{2}{*}{16} & 6 & 141813.6 & 162397.4 & 12.7\% & 3600 s & 142379.2 & 0.40\% & 1125 s
				\\
				& 8 & 160403.1 & 184092.8 & 12.9\% & 3600 s & 164933.3 & 2.75\% & 2964 s
				\\\hline
				\multirow{2}{*}{20} & 6 & 167127.6 & 192935.2 & 13.4\% & 3600 s & 167627.1 & 0.30\% & 1344  s
				\\
				& 8 & 187999.7 & 222681.2 & 15.6\% & 3600 s & 193904.4 & 3.10\% & 3468 s
				\\\hline
			\end{tabular}
			\begin{tablenotes}
				\item Note: LB is the lower bound provided by Gurobi; Obj$_1$ is the objective value of the best incumbent solution of Gurobi; Gap$_1$ is the gap between the LB and Obj$_1$. Obj$_2$ is the objective value of the optimal solution by the batch scheduling algorithm; Gap$_2$ is the gap between LB and Obj$_2$.
			\end{tablenotes}
		\end{threeparttable}	
		\label{tabAP} 			
	\end{table}

	\subsection{Optimization results and sensitivity analyses} \label{subsecCR}
	In this section, we show and interpret the optimization results, which provide insights into the joint management of bus fleet and drive-by sensing objectives.

	\subsubsection{IB and NB schedules}
	
	We begin by analyzing IB and NB schedules generated by the proposed methods to illustrate their coordination in serving the timetables while performing DS tasks. In this section, the number of IBs is set to be 5. We illustrate their schedules from three aspects.
\begin{enumerate}
\item	{\bf Utilization of sensors:} Figs. \ref{figinbrouting} (a) and (b) compare the IB schedules under M2 and M3 respectively. For M2, we see that the 5 IBs are distributed in 5 lines (\# 55, 76, 343, 47, and 54). In the M3 approach, these 5 IBs are scheduled to serve 7 lines through relocation, extending the spatial coverage of M2. Indeed, the grid coverage rates (percentage of grids covered by IB at least once) are 68\% in M2 and 81\% in M3. Moreover, there are 9 wait arcs in M2, and only 2 wait arcs in M3, meaning that inter-line relocations can further increase the utilization of sensing resources. 

\item {\bf IB relocations:} For M3, Line 80 has the most interaction with the other lines (\# 55, 99, 4, 47) in terms IB relocations. Fig. \ref{relocations} (a) explains why: (1) they all have terminals that are close to each other, which facilitates bus relocations; and (2) these lines all extend to different parts of the study area, gaining advantage in spatial coverage. Such an intuitive example demonstrates the elegant balance between relocation costs and sensing quality afforded by the proposed model and the proposed algorithm.

\item {\bf NB relocations:} Fig. \ref{figinbrouting} (c) shows the NB schedules in the M3 approach. Under multi-line scheduling, some NBs are relocated to different lines in coordination with IB scheduling. It can be verified from Figs. \ref{figinbrouting} (b) and (c) that the timetabled trips are fulfilled by IBs and NBs in M3. Fig. \ref{relocations} (b) highlights areas where relocation takes place. 
\end{enumerate}

	\begin{figure}[h!]
		\centering
		\makebox[\textwidth][c]{\includegraphics[width=1\textwidth]{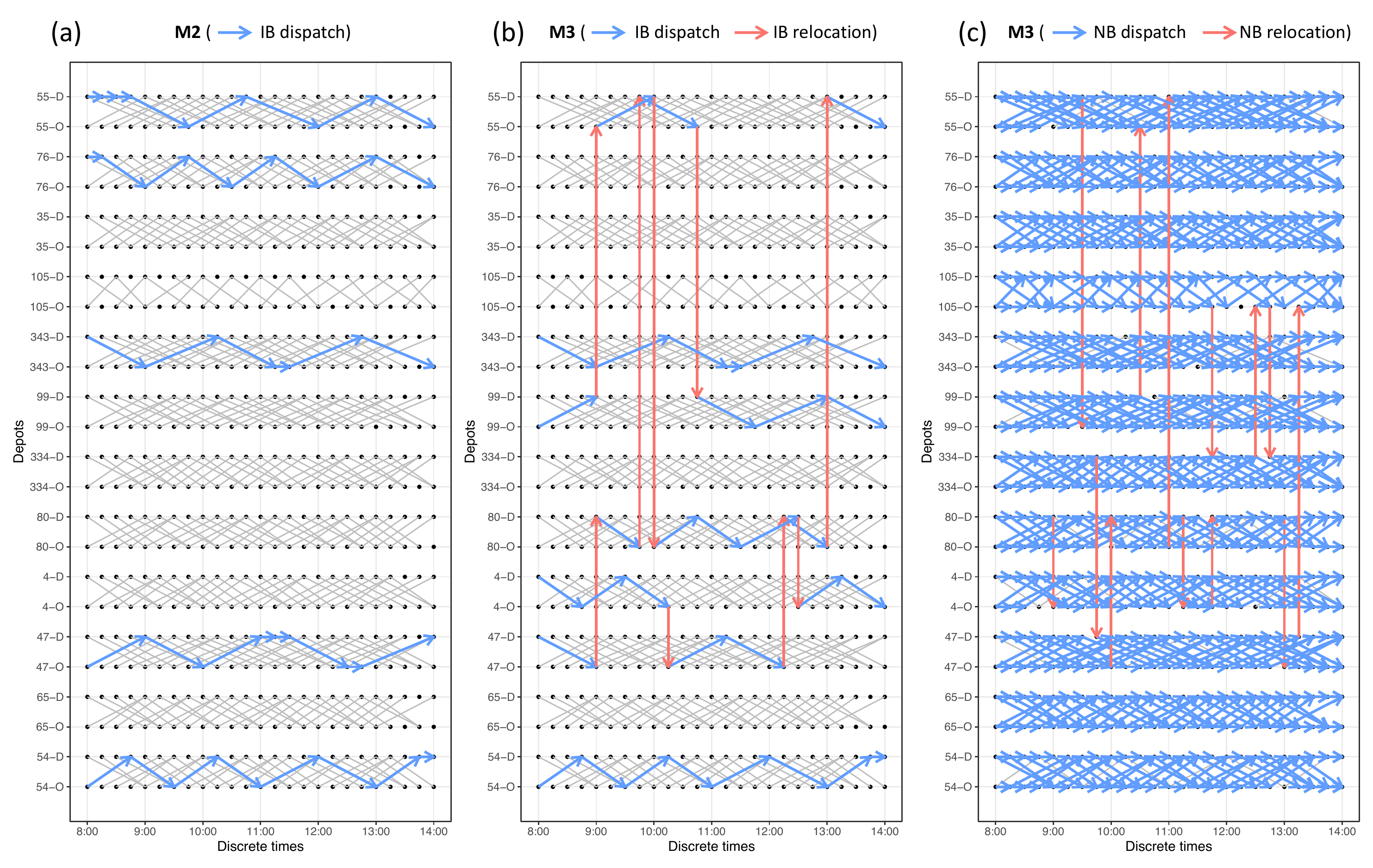}}
		\caption{Visualization of IB and NB schedules. (a) Schedules of 5 IBs under M2; (b) Schedules of 5 IBs under M3; (c) NB schedules under M3.}
		\label{figinbrouting}   
	\end{figure}

	\begin{figure}[h!]
		\centering
		\makebox[\textwidth][c]{\includegraphics[width=1\textwidth]{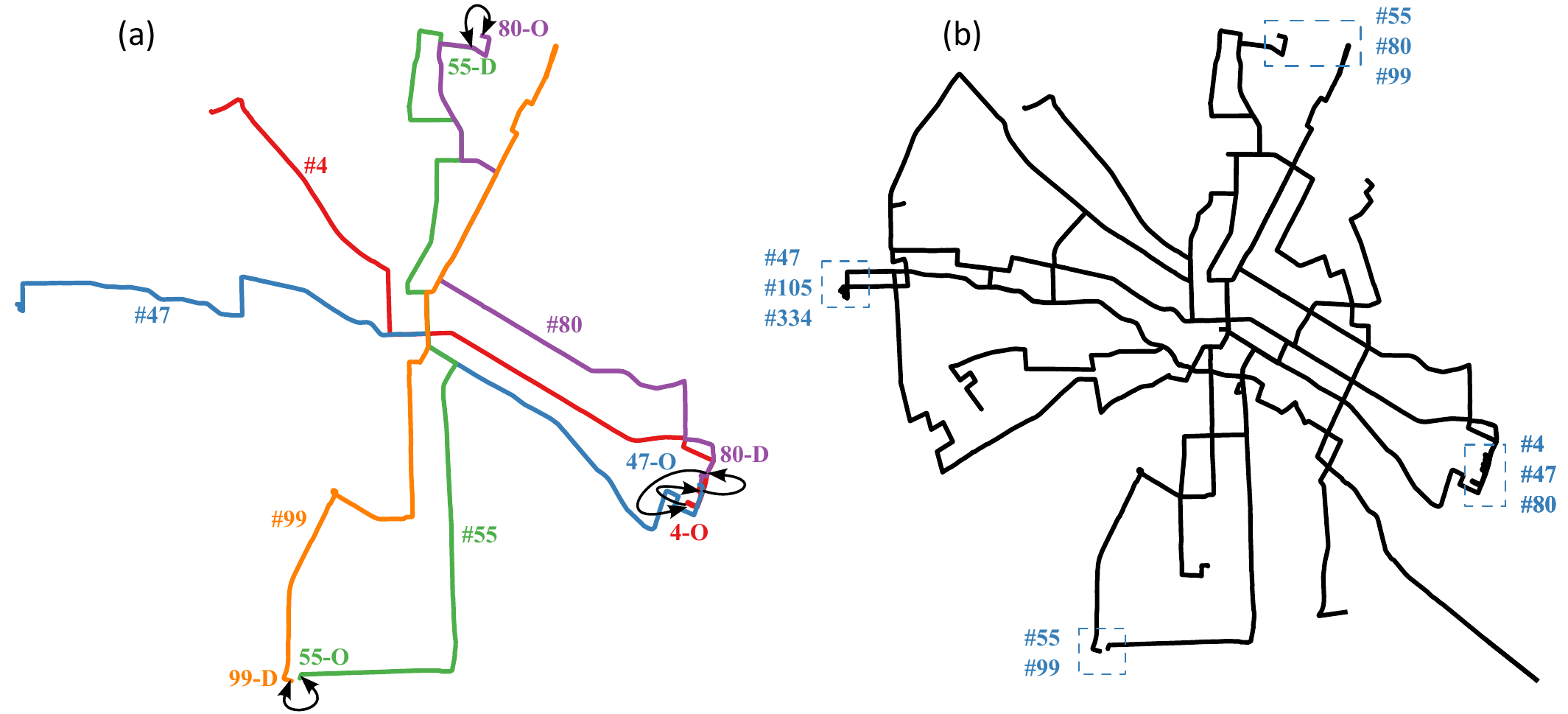}} 
		\caption{Spatial visualization of bus relocation. (a) IB relocations (dark arrows) among a few lines; (b) IB and NB relocations (dashed rectangles) among all the lines.}
		\label{relocations} 
	\end{figure}

	\subsubsection{Impact of active IB scheduling}\label{subsubimpact}

	Table \ref{tabAS} demonstrates the impact of active IB scheduling on the system performance, in terms of fixed cost (equivalent to number of buses required), operational cost (primarily related to bus relocations), total cost, grid coverage (\%), and sensing score. To differentiate the impact of IB fleet size, we report results for 5, 10, 15, and 20 IBs. 
	
	Firstly, the active single-line IB scheduling approach (M2) yields the same fixed and operational costs as M1. This is expected as M2 uses the same fleet as M1 without inter-line relocation. Nevertheless, through active single-line scheduling, the sensing score of M2 for 5, 10, 15 and 20 IBs increases by 20.5\%, 12.0\%, 13.8\%, and 12.2\%, respectively. We observe that the grid coverage percentage of M2 is not necessarily improved compared with M1, which is due to the pursuit of the sensing score, leaving some grids with low weights unvisited. Such a situation, if and when deemed undesirable, can be remedied by adjusting the weights, or by further suppressing the marginal gain in the effective sensing times.

	Secondly, the M3 approach (active multi-line bus scheduling)  reduces the fleet size by 2\% compared to M1 and M2, which confirms existing literature that show high bus utilization enabled by multi-line scheduling \citep{sayarshad2020optimizing}. The total (fixed+operational) cost of M3, despite the higher operational cost due to bus relocations, is roughly 1\% less than M1/M2. Moreover, M3 considerably improves the sensing quality over M1 and M2, because IBs can be relocated to cover grids of multiple lines, and the proposed batch scheduling heuristic guarantees that the IBs are well coordinated. 
	
	An important observation is that, as the IB fleet size increases, the advantage of M3 over M1 or M2 in terms of sensing score decreases. The reason is as follows. For a fixed temporal sensing granularity $\Delta_k$ and effective sensing time function $\hat f(\cdot)$, a bus line can be saturated with certain number of IBs, past which point more coverage would be redundant, which is a direct consequence of the diminishing marginal sensing gain. Therefore, higher number of IBs bring many bus lines to near-saturation, and sensor circulation enabled by M3 brings limited benefits. On the other hand, as the IB fleet size increases, the improvement of M2 over M1 remains relatively steady (around 12\%-13\%). This is because in both cases,  the IBs are restricted to single lines, which experience similar levels of saturation when the fleet size increases, and the 12\%-13\% increase is mainly attributed to intra-line scheduling.

	\begin{table}[h!]	
		\centering
		\caption{The effect of active bus scheduling on system performances ($\Delta_k=2$). }
		\resizebox{\textwidth}{!}{%
		\begin{threeparttable}
				\begin{tabular}{lllllllll}
					\toprule
					IB fleet  &  \multirow{2}{*}{Approach}   & Buses     & Operational  & Total  & Grid & Sensing
					\\ 
					size &   & required & cost* & cost & coverage & score $*$
					\\
					\midrule
					\multirow{3}{*}{5}  & M1  & 100      & 38724              & 124324 & 68\%       & 0.78   \\
					& M2  & 100      & 38724 ($+0.0\%$)   & 124324 ($+0.0\%$) & 68\%       & 0.94 (+20.5\%)   \\
					& M3  & 98       & 39248 ($+1.4\%$)   & 123136 ($-1.0\%$) & 81\%       & 1.03 (+32.1\%)  
					\\\hline
					\multirow{3}{*}{10} & M1   & 100     & 38724              & 124324 & 98\%       & 1.25   \\
					& M2   & 100     & 38724 ($+0.0\%$)   & 124324 ($+0.0\%$) & 94\%       & 1.40 (+12.0\%)  \\
					& M3   & 98      & 39191 ($+1.2\%$)   & 123079 ($-1.0\%$) & 98\%       & 1.50 (+20.0\%)  
					\\\hline
					\multirow{3}{*}{15} & M1   & 100     & 38724              & 124324 & 100\%              & 1.38   \\
					& M2   & 100     & 38724 ($+0.0\%$)   & 124324 ($+0.0\%$) & 98\%       & 1.57 (+13.8\%)  \\
					& M3   & 98      & 39251 ($+1.4\%$)   & 123139 ($-1.0\%$) & 100\%      & 1.65 (+19.6\%)  
					\\\hline
					\multirow{3}{*}{20} & M1   & 100     & 38724              & 124324 & 100\%              & 1.47   \\
					& M2   & 100     & 38724 ($+0.0\%$)   & 124324 ($+0.0\%$) & 100\%      & 1.65 (+12.2\%)  \\
					& M3   & 98      & 39048 ($+0.8\%$)   & 122936 ($-1.1\%$) & 100\%      & 1.71 (+16.3\%)  \\ 
					\bottomrule
			\end{tabular}
		\begin{tablenotes}
			\item Note: columns marked with $*$ present relative improvement over M1 in the parentheses.
		\end{tablenotes}
		\end{threeparttable}	 }
		\label{tabAS}    
	\end{table}

	Next, we provide a visualization of the spatial-temporal distribution of  sensing score in Fig. \ref{distributions}, for an IB fleet size of 5, a planning horizon of 6 hours, and temporal sensing granularity $\Delta_k=2$ (hrs). Each row corresponds to an approach (M1-M3), and each column corresponds to a two-hour sensing period. We can see that, with active scheduling (M1), the sensing resources are vastly underutilized (white or light red in most areas), rendering a 68\% grid coverage and sensing score of 0.78. When the IBs are actively scheduled in a single-line manner (M2), the coverage spreads more evenly over time, generating a higher sensing score of 0.94. However, area covered by M2 is limited to the 5 bus lines, amounting to 68\% of all grids. Finally, when IBs are actively scheduled over the network with multi-line relocation (M3), the 5 IBs can cover 81\% of all grids, and more grids have received the largest sensing score (1.732), raising the total to 1.03. These visualization results suggest that, when the IBs are relatively sparse, their active scheduling tends to have a positive impact on the sensing quality, and such impact grows when the IBs are circulated in the network.

	\begin{figure}[h!] 
		\centering
		\includegraphics[width=\textwidth]{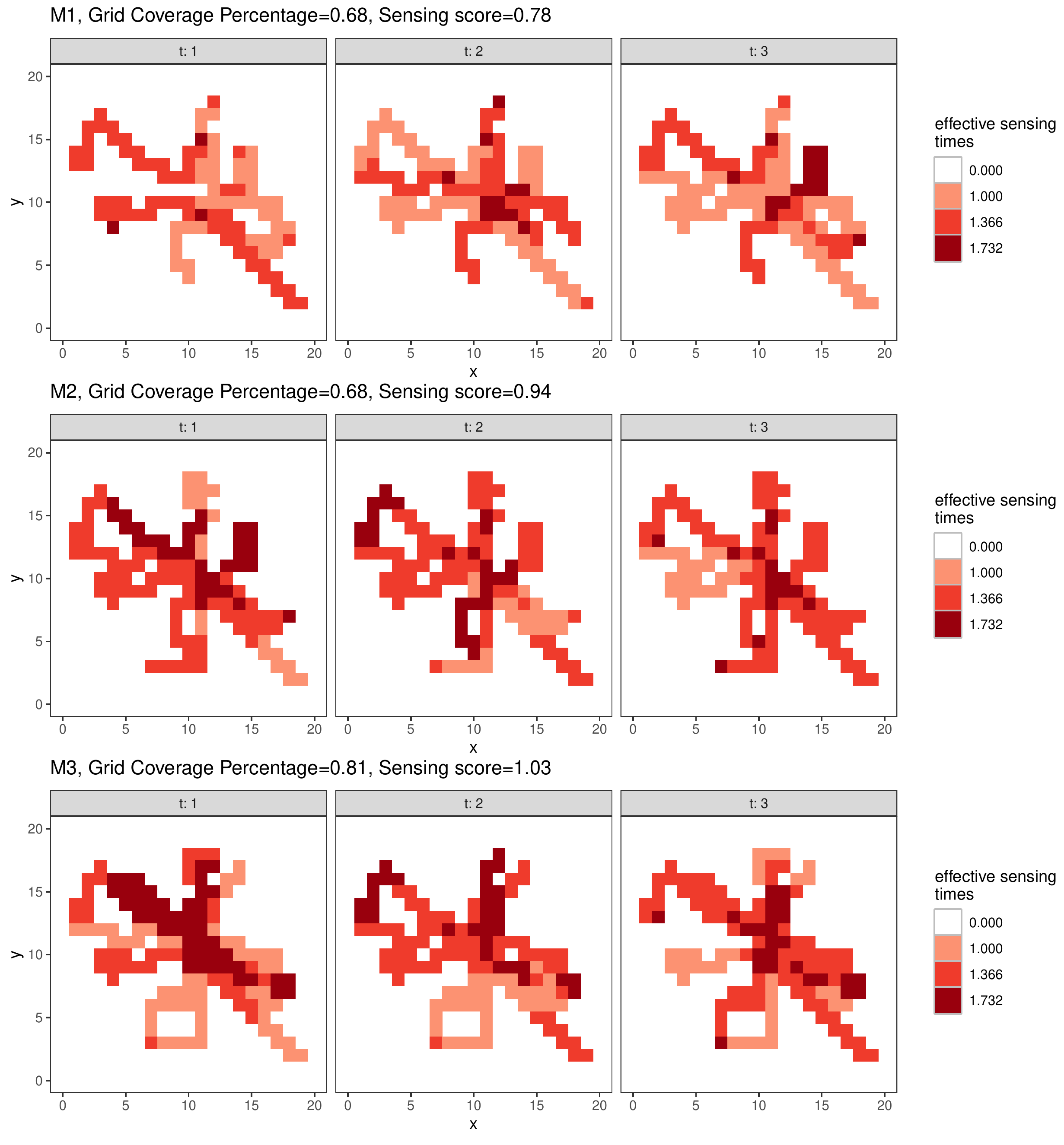}
		\caption{Spatial-temporal distributions of the sensing score for scheduling 5 IBs ($t=1,2,3$ represent the three sensing periods respectively)}
		\label{distributions}  
	\end{figure}

	\subsubsection{Impact of the temporal sensing granularity}\label{subsubsecimpact}
	
	We note that the temporal sensing granularity $\Delta_k$ plays a vital role not only in the quantitative assessment of the sensing quality, but also in determining the optimal bus schedules. To understand its impact on the sensing score, we invoke a test instance with 5 or 15 IBs,  a time horizon of 6 hrs, and $\Delta_k=1, 3, 6$ hrs, whose results are reported in Table \ref{tabML}. These three values of $\Delta_k$ represent decreasing importance of timeliness, or data requirement, on the sensing part. The following observations are made:
	\begin{itemize}
	\item Regardless of the IB fleet size, the advantage of M2 over M1 quickly diminishes as $\Delta_k$ increases. This is due to the fact that both M1 and M2 restrict IBs to fixed lines, and the role played by active scheduling (M2) becomes insignificant for large $\Delta_k$. 
	
	\item When the IB fleet size is small (5), the advantage of M3 over M2 grows with $\Delta_k$, because the larger time frame (e.g. 6 hrs) allows the IBs to be relocated to multiple lines, collecting higher sensing socres throughout the network, whereas lines equipped with IBs in M2 are saturated, irrespective of scheduling maneuvers. 
	
	\item On the contrary, when the IB fleet size is large (15), the advantage of M3 over M2 diminishes with increasing $\Delta_k$, because in this case each bus line is almost guaranteed one IB, which can collect the maximum effective sensing time (1.73) within 6 hrs, making multi-line scheduling unnecessary. 
	\end{itemize}

	\begin{table}[h!]	
		\caption{The impact of the temporal sensing granularity $\Delta_k$. }
		\resizebox{\textwidth}{!}{
		\begin{threeparttable}
				\begin{tabular}{llllllllll}
					\toprule
					IB fleet & \multirow{2}{*}{$\Delta_k$}   &  \multirow{2}{*}{Method}   & Buses     & Operational  & Total  & Grid & Sensing
					\\ 
					size & &   & required & cost* & cost & coverage & score $*$
					\\
					\midrule
					\multirow{9}{*}{5} & \multirow{3}{*}{1 hr}  & M1  & 100      & 38724              & 124324              & 68\%       & 0.48   \\
					&  & M2  & 100      & 38724 (+0.0\%)              & 124324 (+0.0\%)             & 65\%       & 0.63 (+31.3\%)   \\
					&  & M3  & 98      & 39168 ($+1.1\%$)   & 123056 ($-1.0\%$)    & 88\%       & 0.68 (+41.7\%)  
					\\\cline{2-8}
					& \multirow{3}{*}{3 hrs} & M1   & 100     & 38724              & 124324              & 68\%       & 0.97  \\
					&  &  M2   & 100     & 38724 (+0.0\%)             & 124324 (+0.0\%)             & 68\%       & 1.12 (+15.5\%)  \\
					&	& M3   & 98     & 39288 ($+1.5\%$)   & 123176 ($-0.9\%$)    & 98\%       & 1.31 (+35.1\%)  
					\\\cline{2-8}
					& \multirow{3}{*}{6 hrs} & M1   & 100     & 38724              & 124324              & 68\%       & 1.16   \\
					&  & M2   & 100     & 38724 (+0.0\%)             & 124324 (+0.0\%)             & 68\%       & 1.16 (+0.0\%)   \\
					&  	& M3   & 98     & 39248 ($+1.4\%$)   & 123136 ($-1.0\%$)    & 95\%       & 1.60 (+37.9\%)  \\
					\midrule
					\multirow{9}{*}{15} & \multirow{3}{*}{1 hr}  & M1  & 100      & 38724              & 124324              & 100\%       & 0.98   \\
					&  & M2  & 100      & 38724 (+0.0\%)             & 124324 (+0.0\%)             & 88\%       & 1.12 (+14.3\%)   \\
					&  & M3  & 98      & 39168 ($+1.1\%$)   & 123056 ($-1.0\%$)    & 100\%       & 1.32 (+34.7\%)  
					\\\cline{2-8}
					& \multirow{3}{*}{3 hrs} & M1   & 100     & 38724              & 124324              & 100\%       & 1.57   \\
					&  &  M2   & 100     & 38724 (+0.0\%)             & 124324 (+0.0\%)             & 100\%       & 1.68 (+7.0\%)  \\
					&	& M3   & 98     & 39048 ($+0.8\%$)   & 122936 ($-1.1\%$)    & 100\%       & 1.73 (+10.2\%)  
					\\\cline{2-8}
					& \multirow{3}{*}{6 hrs} & M1   & 100     & 38724              & 124324              & 100\%       & 1.73   \\
					&  & M2   & 100     & 38724 (+0.0\%)             & 124324 (+0.0\%)             & 100\%       & 1.73 (+0.0\%)  \\
					&  	& M3   & 98     & 39068 ($+0.9\%$)   & 122956 ($-1.1\%$)    & 100\%       & 1.73 (+0.0\%)  \\
					\bottomrule 
			\end{tabular} 
		\end{threeparttable}}
		\begin{tablenotes}
			\item Note: columns marked with $*$ present relative improvement over M1 in the parentheses.
		\end{tablenotes}	 		
		\label{tabML}  
	\end{table}

\subsubsection{Impact of bus relocation costs}\label{subsubsecimpactrc}

It is well understood that the overall sensing efficacy of the IB fleet depends considerably on inter-line relocation. Therefore, the relocation cost plays a vital role in the final outcome of the schedules, which we will investigate here with a sensitivity analysis. 

We consider the case with 6 bus lines, 5 instrumented buses, and a time horizon of 6 hours. We perturb the fixed bus relocation cost (20 RMB in all other experiments) from 0 to 500 (RMB), as shown in Fig. \ref{figrcsensitivity}, and observe the algorithm output. It can be seen that the overall model objective \eqref{FP1} increases with the relocation cost; in fact, the overall operational cost increases with relocation cost, and the sensing score decreases because bus relocations decline. Moreover, the sensing score curve is piecewise constant, meaning that the number of relocations remain constant for certain relocation costs, and the jump occurs when the relocation number changes. Indeed, when the relocation cost is beyond 150 RMB, no relocations exist in the solution. 

\begin{figure}[h!]
		\centering
	\includegraphics[width=\textwidth]{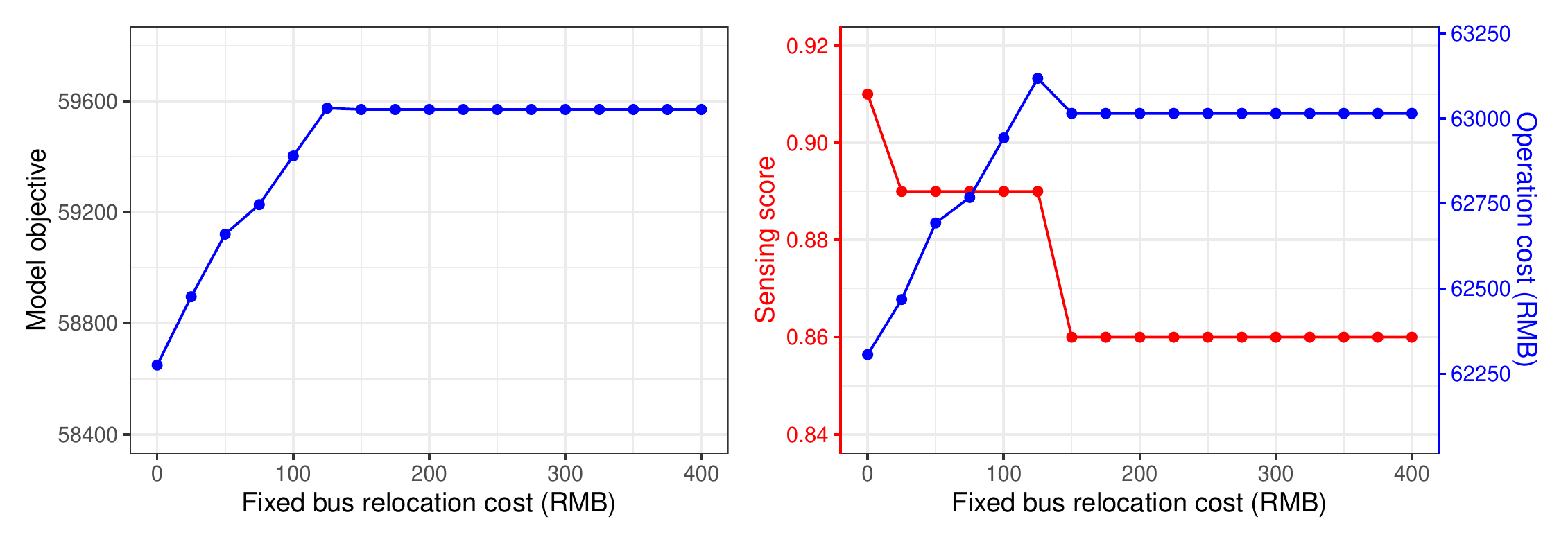}
		\caption{Sensitivity analysis of fixed relocation cost (the cost used in all the other experiments is 20 RMB).}
		\label{figrcsensitivity} 
	\end{figure}

	\subsection{Trade-off analysis of sensing quality and cost}\label{subsecPD}

	The parameter $\delta$ from Eqn \eqref{FP1} concerns with the importance of drive-by sensing quality relative to operational costs. To understand the response of the objectives to $\delta$, its values are populated ranging from 0 to 20000 with an increment of 250, followed by a series of optimizations \eqref{FP1}-\eqref{FP11}. Other settings in this experiment are: 20 IBs, a planning horizon of 6 hours and $\Delta_k=2$ (hrs). Fig. \ref{parameterDelta} shows the optimization results, from which we see that higher values of $\delta$ raise the sensing score because of the higher weights placed on the sensing quality, but this also increases the costs because more relocations are needed to mobilize the sensors. We also observe that, as the cost increases, the marginal gain of the sensing score reduces. This is because the sensing score of a timed grid is bounded and will stop growing once the grid's coverage is saturated with $\geq$ 3 times; see \eqref{example}. Finally, $\delta=4000$ achieves a satisfactory level of the sensing score, past which point the trade-off with the cost becomes quite inefficient. For this reason, we choose $\delta=4000$ in Sections \ref{subsecAP}-\ref{subsecCR}. 
	
		\begin{figure}[h!]
		\centering
		\makebox[\textwidth][c]{\includegraphics[width=0.7\textwidth]{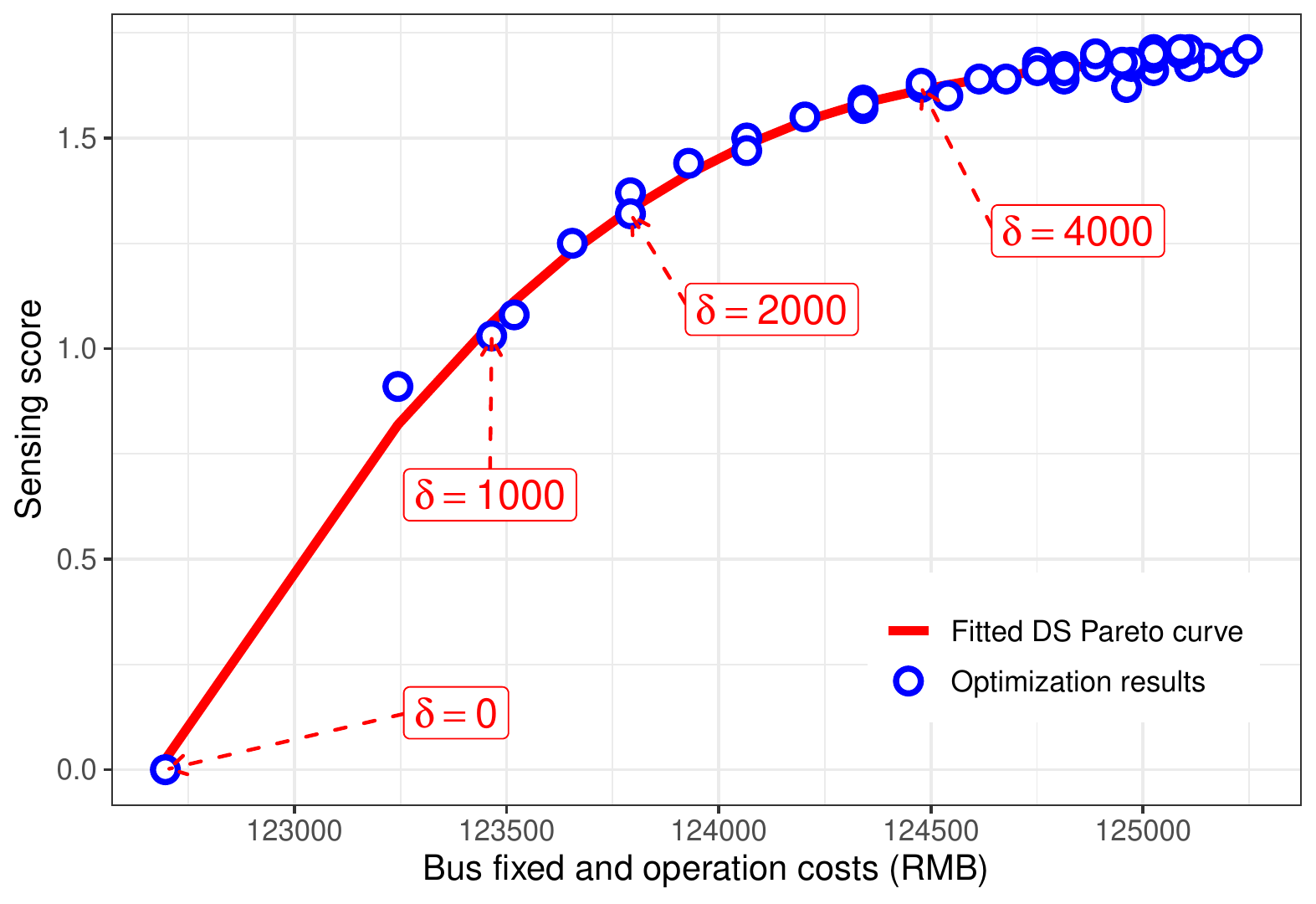}}%
		\caption{Trade-off between cost and sensing quality. Each blue circle represents one $\delta$ value; the red curve is the fitted Pareto efficiency.}
		\label{parameterDelta} 
	\end{figure}  
	
	Fig. \ref{parameterDelta} visualizes the trade-off between the sensing score and bus-related costs, when the number of IBs (sensors) is given. In Fig. \ref{figCBSTWEST}, we fix $\delta=4000$, and show the relationship between sensor investment (IB fleet size)\footnote{The sensor investment is expressed in terms of the IB fleet size as we naturally assume fixed unit cost for sensor procurement, installation and maintenance. In this way, the findings here are generalizable regardless of the type of the sensor.} and sensing score. This figure reflects the `return on investment' on part of the sensing agency, when different approaches are employed by the bus operator. We can also see the potential saving in sensor investment through active bus scheduling (M2/M3 vs. M1), by comparing the numbers of sensors (or IBs) required to achieve the same level of sensing score, as summarized in Table \ref{tabRoI}. It can be seen that to achieve satisfactory sensing quality (sensing score $\geq 1.0$), the maximum saving can be over 33.3\% compared to the benchmark M1 (i.e. optimal sensor deployment without active scheduling).

	\begin{figure}[h!]
		\centering
		\makebox[\textwidth][c]{\includegraphics[width=0.7\textwidth]{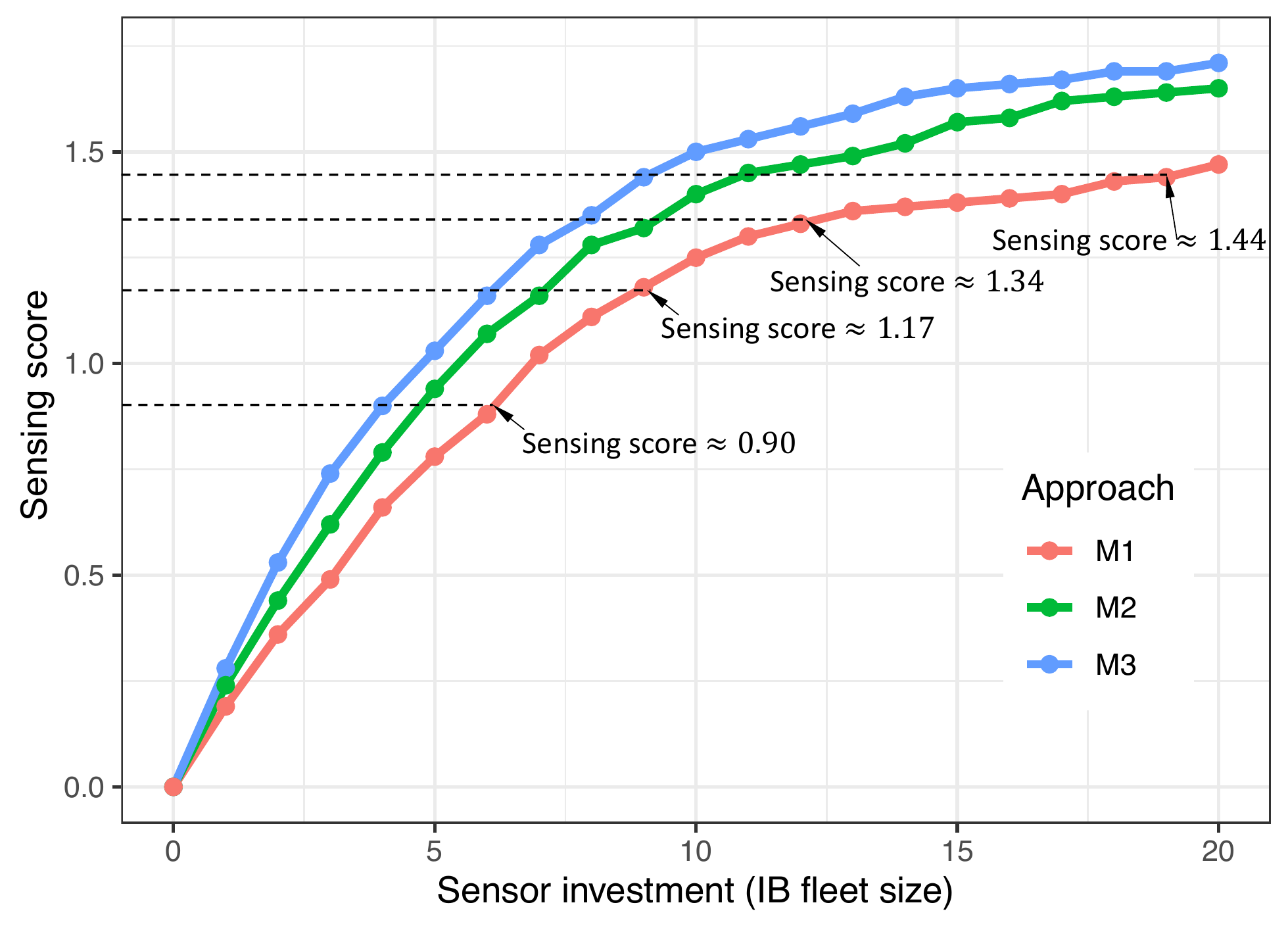}}%
		\caption{Trade-off between sensor investment and sensing quality.}
		\label{figCBSTWEST} 
	\end{figure}
	
		\begin{table}[h!]
	\caption{Number of sensors required to achieve a similar level of sensing score.}
	\centering
	\begin{tabular}{c|ccc|c}
	\hline
	Sensing score & M1 & M2 & M3 & Saving (M3 over M1)
	\\\hline
	0.90 & 6 & 5 & 4 & 33.3\%
	\\
	1.17 & 9 & 7 & 6 & 33.3\%
	\\
	1.34 & 12 & 9 & 8 & 33.3\%
	\\
	1.44 & 19 & 11 & 9 & 52.6\%
	\\
	\hline
	\end{tabular}
	\label{tabRoI}
	\end{table}

	\section{Conclusion}\label{secConclude}

	This paper explores the drive-by sensing power of bus fleets via multi-line scheduling, which advances state of the art from strategic sensor deployment to tactical/operational bus maneuvers. The underlying premise is to enable flexible and sensing-oriented circulation of instrumented buses in the network through a centralized optimization procedure. To ensure the operational feasibility of this approach, we propose a multi-line drive-by sensing bus assignment \& scheduling problem, which not only features bus service, wait, pull-out/in, and relocation activities, but also ensures that all timetabled trips are served by a mixed fleet of normal and instrumented buses. With such constraints, the problem aims to simultaneously maximize sensing utility, expressed as a weighted space-time effective sensing times (which we coin as `sensing score'), and minimize fixed and operational costs of buses. Two variations of this model are considered and compared: single-line scheduling and multi-line scheduling.

	The problem is formulated as a nonlinear integer programming model based on a time-expanded network. The model is linearized through piecewise linear approximation, followed by a batch scheduling heuristic that efficiently solve large-scale instances of the problem, which is benchmarked by Gurobi. Specific findings and recommendation for practice are as follows: 
	\begin{enumerate}
	    \item The proposed solution scheme yields good solution quality, with a gap less than 3.1\% from the lower bound generated by Gurobi, and perform more efficiently than Gurobi for large instances.
	    \item The schedules generated within the model produce several instances of relocations in the multi-line scenario, all with insignificant relocation distances and costs. The single-line scenario yields no intra-line relocations as they are highly cost-ineffective. 
	    \item Both single-line and multi-line scheduling (M2 \& M3) considerably improve the sensing quality compared to M1, by 12.0\%-20.5\% and 16.3\%-32.1\%, respectively. In M3, thanks to multi-line scheduling the fleet size is smaller than M1 and M2, rendering a lower total cost.
	    \item The advantage of M2 and M3 over M1: (1) declines with higher IB saturation rate as the marginal gain of active scheduling diminishes; (2) declines as the temporal sensing granularity $\Delta_k$ increases, except when the sensor saturation is low, in which case M3 dominates M2 and M1. Our recommendation regarding the use of active scheduling is qualitatively illustrated in Table \ref{tabrecomm}.
	    \item The trade-off between sensing objectives and bus operational costs is central to bus DS with active operational maneuvers. In case of multi-line scheduling with given number of sensors (IBs), prioritized sensing objective tends to increase operational costs due to more frequent relocation. On the other hand, from the perspective of the sensing agency, the `return on investment' of sensor deployment could benefit greatly from active bus scheduling, with potential budget savings of over 33\% in our case study. 
	\end{enumerate}
	
	\begin{table}[h!]
		    \caption{Overall recommendation for using active bus scheduling in DS tasks.}
	    \centering
	    \begin{tabular}{cc|c|c|c|}
	    \cline{3-5}
	  &  &  \multicolumn{3}{c|}{Temporal sensing granularity ($\Delta_k$)}
	    \\\cline{3-5}
	   &      & ~~~Small~~~ & ~Medium~ & ~Large ~
	         \\\hline
	   \multicolumn{1}{|c|}{\multirow{2}{*}{Sensor}} &   Low  & \checkmark\checkmark\checkmark & \checkmark\checkmark\checkmark & \checkmark\checkmark\checkmark
	       \\\cline{2-5}
	   \multicolumn{1}{|c|}{\multirow{2}{*}{saturation}} &   Medium & \checkmark\checkmark &\checkmark\checkmark &\checkmark\checkmark
	       \\\cline{2-5}
	   \multicolumn{1}{|c|}{\multirow{2}{*}{}} &   High & \checkmark\checkmark&\checkmark &
	       \\\hline
	    \end{tabular}
	    \label{tabrecomm}
	\end{table}

	In Table \ref{tabrecomm}, sensor saturation refers to the average number of sensors per bus line. Temporal sensing granularity $\Delta_k$ refers to the time interval during which the effective sensing times are calculated. Sensing applications with higher sampling frequency requirements (such as air quality, traffic conditions, parking availability) have low $\Delta_k$, whereas tasks such as road surface monitoring have large $\Delta_k$

    Future extensions of this work include: (1) solution algorithms for large-scale applications (over 100 lines); and (2) coordination with other types of host vehicles such as taxis or dedicated vehicles to further boost the sensing powers of the fleets.

    \section*{Acknowledgement}
    This work is supported by the National Natural Science Foundation of China through grant 72071163, the Natural Science Foundation of Sichuan through grants 2022NSFSC0474 and 2022NSFSC1909, and the Fundamental Research Funds for the Central Universities through grants 2682021CX054 and 2682021CX056.

\begin{appendices}		
		
\section{Sensitivity analysis on parameter $w$ within the batch scheduling algorithm}\label{appendixA}
		
In Section \ref{subsecAD}, Step 1 of the proposed algorithm introduces a weighting parameter $\omega$ to indicate the relative importance of drive-by sensing objective over trip coverage, which appears in Eqn \eqref{SA27}. Here, we perform a sensitivity analysis of $\omega$ to understand its impact on the objective. The test involves a fleet of 5, 10, 15 and 20 IBs over 12 bus routes for a planning horizon of 6 hours, while the NB fleet size is determined internally by the optimization procedure, such that all the timetabled trips are served.

The test range of $\omega$ is set as $[0.1,\,3.0]$ with an increment of 0.1. Fig. \ref{figW} shows the corresponding optimality gap between solutions from the proposed algorithm and Gurobi. First, we observe that the impact of $\omega$ on the optimality gap is non-monotone and non-regular, which may be caused by unobserved factors in the numerical procedure.  This also suggest non-dominant position of either objective over the other. Moreover, the overall gap decreases as the IB fleet size increases, which suggest that the heuristic generates good solutions as the sensor saturation is high.  We further see that, for all test cases the proposed heuristic yields good-quality solutions with gaps between $-0.5\%$ and $2\%$.

		\begin{figure}[h!]
			\centering
			\includegraphics[width=0.8\textwidth]{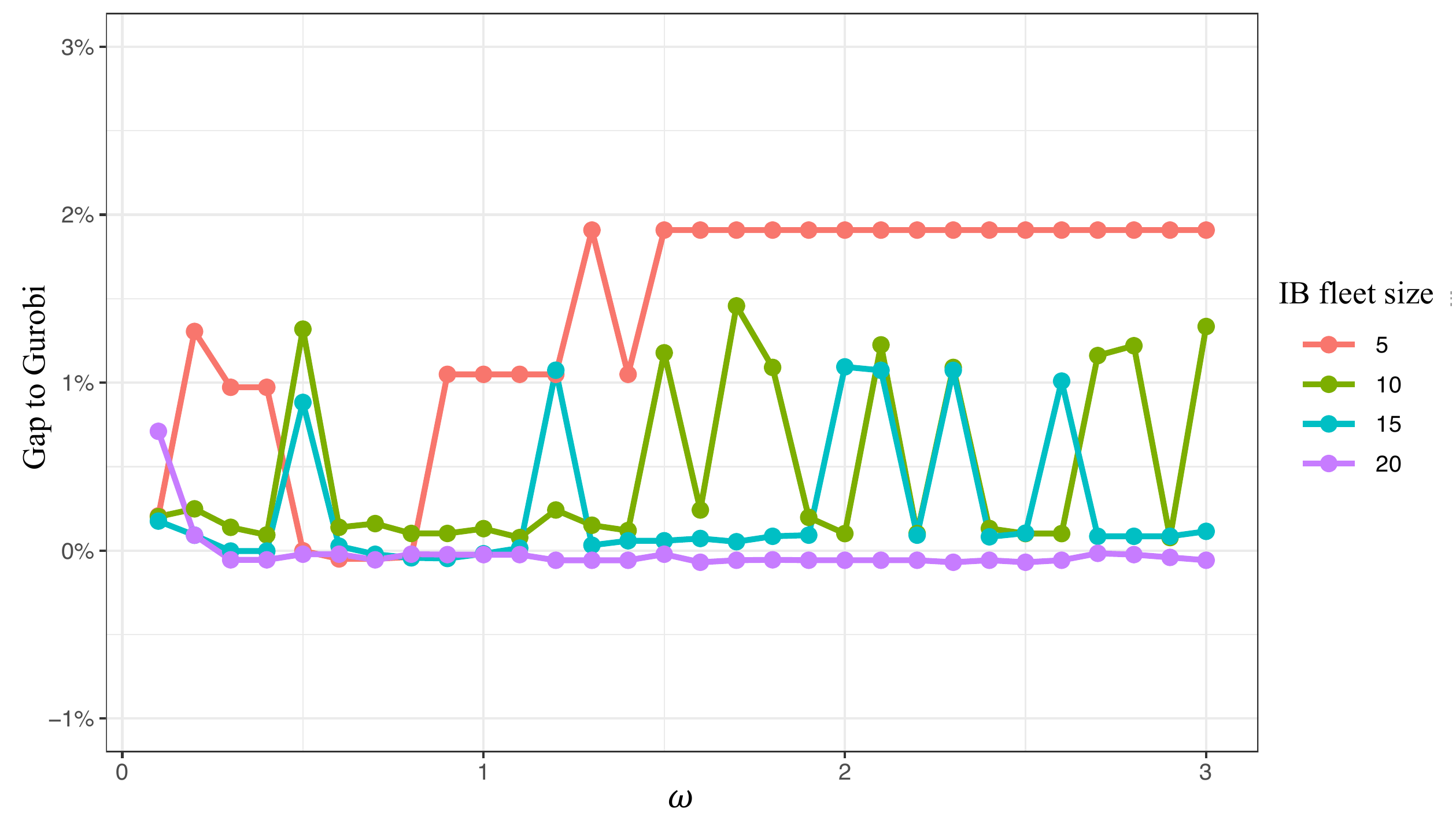}
			\caption{The effect of parameter $w$ on algorithm performance.}
			\label{figW}
		\end{figure}

		\section{A network with 20 bus lines.}\label{appendixB}
		\begin{figure}[h!]
			\centering
			\includegraphics[width=0.8\textwidth]{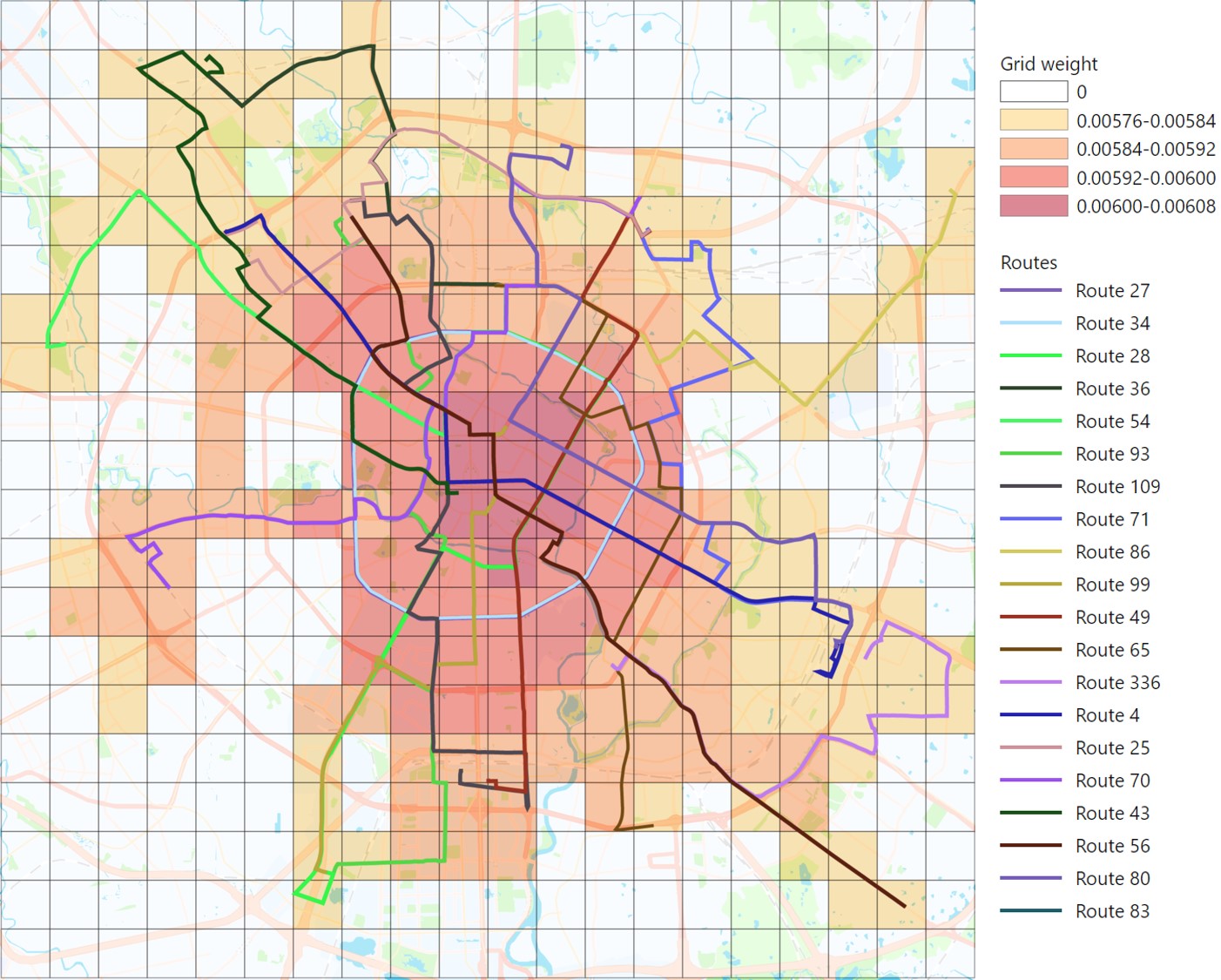}
			\caption{The study area with 20 bus lines and a mesh structure of 1 km $\times$ 1 km}
		\end{figure}

	\end{appendices}

	%\bibliographystyle{elsarticle-harv}
	%\bibliography{ref}

\begin{thebibliography}{99}
\bibitem[{Ali and Dyo(2017)}]{ali2017coverage}
\bibinfo{author}{Ali, J.}, \bibinfo{author}{Dyo, V.}, \bibinfo{year}{2017}.
\newblock \bibinfo{title}{Coverage and mobile sensor placement for vehicles on
  predetermined routes: A greedy heuristic approach}, in:
  \bibinfo{booktitle}{14th International Joint Conference on e-Business and
  Telecommunications (ICETE 2017)}, \bibinfo{publisher}{SCITEPRESS--Science and
  Technology Publications}.
  
  
  \bibitem[{Anjomshoaa et~al.(2018)Anjomshoaa, Duarte, Rennings, Matarazzo,
  deSouza and Ratti}]{anjomshoaa2018city}
\bibinfo{author}{Anjomshoaa, A.}, \bibinfo{author}{Duarte, F.},
  \bibinfo{author}{Rennings, D.}, \bibinfo{author}{Matarazzo, T.J.},
  \bibinfo{author}{deSouza, P.}, \bibinfo{author}{Ratti, C.},
  \bibinfo{year}{2018}.
\newblock \bibinfo{title}{City scanner: Building and scheduling a mobile
  sensing platform for smart city services}.
\newblock \bibinfo{journal}{IEEE Internet of things Journal}
  \bibinfo{volume}{5}, \bibinfo{pages}{4567--4579}.
  
  
  \bibitem[{Asprone et~al.(2021)Asprone, Di~Martino, Festa and
  Starace}]{asprone2021vehicular}
\bibinfo{author}{Asprone, D.}, \bibinfo{author}{Di~Martino, S.},
  \bibinfo{author}{Festa, P.}, \bibinfo{author}{Starace, L.L.L.},
  \bibinfo{year}{2021}.
\newblock \bibinfo{title}{Vehicular crowd-sensing: a parametric routing
  algorithm to increase spatio-temporal road network coverage}.
\newblock \bibinfo{journal}{International Journal of Geographical Information
  Science} \bibinfo{volume}{35}, \bibinfo{pages}{1876--1904}.
  
  
  
  
  
  \bibitem[{Chen et~al.(2020)Chen, Xu, Han, Fu, Pi, Joe-Wong, Li, Zhang, Noh and
  Zhang}]{chen2020pas}
\bibinfo{author}{Chen, X.}, \bibinfo{author}{Xu, S.}, \bibinfo{author}{Han,
  J.}, \bibinfo{author}{Fu, H.}, \bibinfo{author}{Pi, X.},
  \bibinfo{author}{Joe-Wong, C.}, \bibinfo{author}{Li, Y.},
  \bibinfo{author}{Zhang, L.}, \bibinfo{author}{Noh, H.Y.},
  \bibinfo{author}{Zhang, P.}, \bibinfo{year}{2020}.
\newblock \bibinfo{title}{Pas: Prediction-based actuation system for city-scale
  ridesharing vehicular mobile crowdsensing}.
\newblock \bibinfo{journal}{IEEE Internet of Things Journal}
  \bibinfo{volume}{7}, \bibinfo{pages}{3719--3734}.
  
  
  
  \bibitem[{Chen et~al.(2017)Chen, Lv, Guo, Zhou and Xu}]{chen2017trajectory}
\bibinfo{author}{Chen, Y.}, \bibinfo{author}{Lv, P.}, \bibinfo{author}{Guo,
  D.}, \bibinfo{author}{Zhou, T.}, \bibinfo{author}{Xu, M.},
  \bibinfo{year}{2017}.
\newblock \bibinfo{title}{Trajectory segment selection with limited budget in
  mobile crowd sensing}.
\newblock \bibinfo{journal}{Pervasive and Mobile Computing}
  \bibinfo{volume}{40}, \bibinfo{pages}{123--138}.
  
  
  
  
  
  
  
  \bibitem[{Cruz et~al.(2020)Cruz, Couto, Costa, Fladenmuller and
  de~Amorim}]{cruz2020per}
\bibinfo{author}{Cruz, P.}, \bibinfo{author}{Couto, R.S.},
  \bibinfo{author}{Costa, L.H.M.}, \bibinfo{author}{Fladenmuller, A.},
  \bibinfo{author}{de~Amorim, M.D.}, \bibinfo{year}{2020}.
\newblock \bibinfo{title}{Per-vehicle coverage in a bus-based general-purpose
  sensor network}.
\newblock \bibinfo{journal}{IEEE Wireless Communications Letters}
  \bibinfo{volume}{9}, \bibinfo{pages}{1019--1022}.
  
  
  
  \bibitem[{Cruz~Caminha et~al.(2018)Cruz~Caminha, de~Souza~Couto, Maciel
  Kosmalski~Costa, Fladenmuller and Dias~de Amorim}]{cruz2018coverage}
\bibinfo{author}{Cruz~Caminha, P.H.}, \bibinfo{author}{de~Souza~Couto, R.},
  \bibinfo{author}{Maciel Kosmalski~Costa, L.H.},
  \bibinfo{author}{Fladenmuller, A.}, \bibinfo{author}{Dias~de Amorim, M.},
  \bibinfo{year}{2018}.
\newblock \bibinfo{title}{On the coverage of bus-based mobile sensing}.
\newblock \bibinfo{journal}{Sensors} \bibinfo{volume}{18},
  \bibinfo{pages}{1976}.
  
  
  
  
  \bibitem[{Dauer and Prata(2021)}]{dauer2021variable}
\bibinfo{author}{Dauer, A.T.}, \bibinfo{author}{Prata, B.d.A.},
  \bibinfo{year}{2021}.
\newblock \bibinfo{title}{Variable fixing heuristics for solving multiple depot
  vehicle scheduling problem with heterogeneous fleet and time windows}.
\newblock \bibinfo{journal}{Optimization Letters} \bibinfo{volume}{15},
  \bibinfo{pages}{153--170}.
  
  
  
  
  \bibitem[{Dong et~al.(2015)Dong, Guan, Chen, Guo and Gao}]{dong2015mosaic}
\bibinfo{author}{Dong, W.}, \bibinfo{author}{Guan, G.}, \bibinfo{author}{Chen,
  Y.}, \bibinfo{author}{Guo, K.}, \bibinfo{author}{Gao, Y.},
  \bibinfo{year}{2015}.
\newblock \bibinfo{title}{Mosaic: Towards city scale sensing with mobile sensor
  networks}, in: \bibinfo{booktitle}{2015 IEEE 21st International Conference on
  Parallel and Distributed Systems (ICPADS)}, \bibinfo{organization}{IEEE}. pp.
  \bibinfo{pages}{29--36}.
  
  
  \bibitem[{Du et~al.(2014)Du, Chen, Yang, Lu, Guan and Shen}]{du2014effective}
\bibinfo{author}{Du, R.}, \bibinfo{author}{Chen, C.}, \bibinfo{author}{Yang,
  B.}, \bibinfo{author}{Lu, N.}, \bibinfo{author}{Guan, X.},
  \bibinfo{author}{Shen, X.}, \bibinfo{year}{2014}.
\newblock \bibinfo{title}{Effective urban traffic monitoring by vehicular
  sensor networks}.
\newblock \bibinfo{journal}{IEEE transactions on Vehicular Technology}
  \bibinfo{volume}{64}, \bibinfo{pages}{273--286}.
  
  
  \bibitem[{Eriksson et~al.(2008)Eriksson, Girod, Hull, Newton, Madden and
  Balakrishnan}]{eriksson2008pothole}
\bibinfo{author}{Eriksson, J.}, \bibinfo{author}{Girod, L.},
  \bibinfo{author}{Hull, B.}, \bibinfo{author}{Newton, R.},
  \bibinfo{author}{Madden, S.}, \bibinfo{author}{Balakrishnan, H.},
  \bibinfo{year}{2008}.
\newblock \bibinfo{title}{The pothole patrol: using a mobile sensor network for
  road surface monitoring}, in: \bibinfo{booktitle}{Proceedings of the 6th
  international conference on Mobile systems, applications, and services}, pp.
  \bibinfo{pages}{29--39}.
  
  
  \bibitem[{Fan et~al.(2021)Fan, Zhao, Guo, Jin, Gan and Wang}]{fan2021towards}
\bibinfo{author}{Fan, G.}, \bibinfo{author}{Zhao, Y.}, \bibinfo{author}{Guo,
  Z.}, \bibinfo{author}{Jin, H.}, \bibinfo{author}{Gan, X.},
  \bibinfo{author}{Wang, X.}, \bibinfo{year}{2021}.
\newblock \bibinfo{title}{Towards fine-grained spatio-temporal coverage for
  vehicular urban sensing systems}, in: \bibinfo{booktitle}{IEEE INFOCOM
  2021-IEEE Conference on Computer Communications},
  \bibinfo{organization}{IEEE}. pp. \bibinfo{pages}{1--10}.
  
  
  
  \bibitem[{Gao et~al.(2016)Gao, Dong, Guo, Liu, Chen, Liu, Bu and
  Chen}]{gao2016mosaic}
\bibinfo{author}{Gao, Y.}, \bibinfo{author}{Dong, W.}, \bibinfo{author}{Guo,
  K.}, \bibinfo{author}{Liu, X.}, \bibinfo{author}{Chen, Y.},
  \bibinfo{author}{Liu, X.}, \bibinfo{author}{Bu, J.}, \bibinfo{author}{Chen,
  C.}, \bibinfo{year}{2016}.
\newblock \bibinfo{title}{Mosaic: A low-cost mobile sensing system for urban
  air quality monitoring}, in: \bibinfo{booktitle}{IEEE INFOCOM 2016-The 35th
  Annual IEEE International Conference on Computer Communications},
  \bibinfo{organization}{IEEE}. pp. \bibinfo{pages}{1--9}.
  
  
  
  \bibitem[{Gryech et~al.(2020)Gryech, Ben-Aboud, Guermah, Sbihi, Ghogho and
  Kobbane}]{gryech2020moreair}
\bibinfo{author}{Gryech, I.}, \bibinfo{author}{Ben-Aboud, Y.},
  \bibinfo{author}{Guermah, B.}, \bibinfo{author}{Sbihi, N.},
  \bibinfo{author}{Ghogho, M.}, \bibinfo{author}{Kobbane, A.},
  \bibinfo{year}{2020}.
\newblock \bibinfo{title}{Moreair: a low-cost urban air pollution monitoring
  system}.
\newblock \bibinfo{journal}{Sensors} \bibinfo{volume}{20},
  \bibinfo{pages}{998}.
  
  
  
  \bibitem[{Ibarra-Rojas et~al.(2014)Ibarra-Rojas, Giesen and
  Rios-Solis}]{ibarra2014integrated}
\bibinfo{author}{Ibarra-Rojas, O.J.}, \bibinfo{author}{Giesen, R.},
  \bibinfo{author}{Rios-Solis, Y.A.}, \bibinfo{year}{2014}.
\newblock \bibinfo{title}{An integrated approach for timetabling and vehicle
  scheduling problems to analyze the trade-off between level of service and
  operating costs of transit networks}.
\newblock \bibinfo{journal}{Transportation Research Part B: Methodological}
  \bibinfo{volume}{70}, \bibinfo{pages}{35--46}.
  
  
  \bibitem[{James et~al.(2012)James, Li and Lam}]{james2012sensor}
\bibinfo{author}{James, J.}, \bibinfo{author}{Li, V.O.}, \bibinfo{author}{Lam,
  A.Y.}, \bibinfo{year}{2012}.
\newblock \bibinfo{title}{Sensor deployment for air pollution monitoring using
  public transportation system}, in: \bibinfo{booktitle}{2012 IEEE Congress on
  Evolutionary Computation}, \bibinfo{organization}{IEEE}. pp.
  \bibinfo{pages}{1--7}.
  
  
  \bibitem[{Kaivonen and Ngai(2020)}]{kaivonen2020real}
\bibinfo{author}{Kaivonen, S.}, \bibinfo{author}{Ngai, E.C.H.},
  \bibinfo{year}{2020}.
\newblock \bibinfo{title}{Real-time air pollution monitoring with sensors on
  city bus}.
\newblock \bibinfo{journal}{Digital Communications and Networks}
  \bibinfo{volume}{6}, \bibinfo{pages}{23--30}.
  
  
  
  \bibitem[{Lane et~al.(2008)Lane, Eisenman, Musolesi, Miluzzo and
  Campbell}]{lane2008urban}
\bibinfo{author}{Lane, N.D.}, \bibinfo{author}{Eisenman, S.B.},
  \bibinfo{author}{Musolesi, M.}, \bibinfo{author}{Miluzzo, E.},
  \bibinfo{author}{Campbell, A.T.}, \bibinfo{year}{2008}.
\newblock \bibinfo{title}{Urban sensing systems: opportunistic or
  participatory?}, in: \bibinfo{booktitle}{Proceedings of the 9th workshop on
  Mobile computing systems and applications}, pp. \bibinfo{pages}{11--16}.
  
  
  
  \bibitem[{Lee and Gerla(2010)}]{lee2010survey}
\bibinfo{author}{Lee, U.}, \bibinfo{author}{Gerla, M.}, \bibinfo{year}{2010}.
\newblock \bibinfo{title}{A survey of urban vehicular sensing platforms}.
\newblock \bibinfo{journal}{Computer Networks} \bibinfo{volume}{54},
  \bibinfo{pages}{527--544}.
  
  
  \bibitem[{Li(2014)}]{li2014transit}
\bibinfo{author}{Li, J.Q.}, \bibinfo{year}{2014}.
\newblock \bibinfo{title}{Transit bus scheduling with limited energy}.
\newblock \bibinfo{journal}{Transportation Science} \bibinfo{volume}{48},
  \bibinfo{pages}{521--539}.
  
  
  \bibitem[{Li et~al.(2021)Li, Lo, Huang and Xiao}]{li2021mixed}
\bibinfo{author}{Li, L.}, \bibinfo{author}{Lo, H.K.}, \bibinfo{author}{Huang,
  W.}, \bibinfo{author}{Xiao, F.}, \bibinfo{year}{2021}.
\newblock \bibinfo{title}{Mixed bus fleet location-routing-scheduling under
  range uncertainty}.
\newblock \bibinfo{journal}{Transportation Research Part B: Methodological}
  \bibinfo{volume}{146}, \bibinfo{pages}{155--179}.
  
  
  \bibitem[{Li et~al.(2008)Li, Shu, Li, Huang, Luo and Wu}]{li2008performance}
\bibinfo{author}{Li, X.}, \bibinfo{author}{Shu, W.}, \bibinfo{author}{Li, M.},
  \bibinfo{author}{Huang, H.Y.}, \bibinfo{author}{Luo, P.E.},
  \bibinfo{author}{Wu, M.Y.}, \bibinfo{year}{2008}.
\newblock \bibinfo{title}{Performance evaluation of vehicle-based mobile sensor
  networks for traffic monitoring}.
\newblock \bibinfo{journal}{IEEE transactions on vehicular technology}
  \bibinfo{volume}{58}, \bibinfo{pages}{1647--1653}.
  
  
  
  
  \bibitem[{Liu et~al.(2005)Liu, Brass, Dousse, Nain and
  Towsley}]{liu2005mobility}
\bibinfo{author}{Liu, B.}, \bibinfo{author}{Brass, P.},
  \bibinfo{author}{Dousse, O.}, \bibinfo{author}{Nain, P.},
  \bibinfo{author}{Towsley, D.}, \bibinfo{year}{2005}.
\newblock \bibinfo{title}{Mobility improves coverage of sensor networks}, in:
  \bibinfo{booktitle}{Proceedings of the 6th ACM international symposium on
  Mobile ad hoc networking and computing}, pp. \bibinfo{pages}{300--308}.
  
  
  \bibitem[{Lu et~al.(2022)Lu, Diabat, Li and Yang}]{lu2022combined}
\bibinfo{author}{Lu, C.C.}, \bibinfo{author}{Diabat, A.}, \bibinfo{author}{Li,
  Y.T.}, \bibinfo{author}{Yang, Y.M.}, \bibinfo{year}{2022}.
\newblock \bibinfo{title}{Combined passenger and parcel transportation using a
  mixed fleet of electric and gasoline vehicles}.
\newblock \bibinfo{journal}{Transportation Research Part E: Logistics and
  Transportation Review} \bibinfo{volume}{157}, \bibinfo{pages}{102546}.
  
  
  
  \bibitem[{Ma et~al.(2015)Ma, Zhao and Yuan}]{ma2015opportunities}
\bibinfo{author}{Ma, H.}, \bibinfo{author}{Zhao, D.}, \bibinfo{author}{Yuan,
  P.}, \bibinfo{year}{2015}.
\newblock \bibinfo{title}{Opportunities in mobile crowd sensing}.
\newblock \bibinfo{journal}{Infocommunications Journal} \bibinfo{volume}{7},
  \bibinfo{pages}{32--38}.
  
  
  \bibitem[{Ma et~al.(2008)Ma, Richards, Ghanem, Guo and Hassard}]{ma2008air}
\bibinfo{author}{Ma, Y.}, \bibinfo{author}{Richards, M.},
  \bibinfo{author}{Ghanem, M.}, \bibinfo{author}{Guo, Y.},
  \bibinfo{author}{Hassard, J.}, \bibinfo{year}{2008}.
\newblock \bibinfo{title}{Air pollution monitoring and mining based on sensor
  grid in london}.
\newblock \bibinfo{journal}{Sensors} \bibinfo{volume}{8},
  \bibinfo{pages}{3601--3623}.
  
  
  
  \bibitem[{Marjovi et~al.(2015)Marjovi, Arfire and Martinoli}]{marjovi2015high}
\bibinfo{author}{Marjovi, A.}, \bibinfo{author}{Arfire, A.},
  \bibinfo{author}{Martinoli, A.}, \bibinfo{year}{2015}.
\newblock \bibinfo{title}{High resolution air pollution maps in urban
  environments using mobile sensor networks}, in: \bibinfo{booktitle}{2015
  International Conference on Distributed Computing in Sensor Systems},
  \bibinfo{organization}{IEEE}. pp. \bibinfo{pages}{11--20}.
  
  
  \bibitem[{Messier et~al.(2018)Messier, Chambliss, Gani, Alvarez, Brauer, Choi,
  Hamburg, Kerckhoffs, LaFranchi, Lunden et~al.}]{messier2018mapping}
\bibinfo{author}{Messier, K.P.}, \bibinfo{author}{Chambliss, S.E.},
  \bibinfo{author}{Gani, S.}, \bibinfo{author}{Alvarez, R.},
  \bibinfo{author}{Brauer, M.}, \bibinfo{author}{Choi, J.J.},
  \bibinfo{author}{Hamburg, S.P.}, \bibinfo{author}{Kerckhoffs, J.},
  \bibinfo{author}{LaFranchi, B.}, \bibinfo{author}{Lunden, M.M.}, et~al.,
  \bibinfo{year}{2018}.
\newblock \bibinfo{title}{Mapping air pollution with google street view cars:
  Efficient approaches with mobile monitoring and land use regression}.
\newblock \bibinfo{journal}{Environmental science \& technology}
  \bibinfo{volume}{52}, \bibinfo{pages}{12563--12572}.
  
  
  
  
  \bibitem[{Moawad et~al.(2021)Moawad, Li, Pancorbo, Gurumurthy, Freyermuth,
  Islam, Vijayagopal, Stinson and Rousseau}]{moawad2021real}
\bibinfo{author}{Moawad, A.}, \bibinfo{author}{Li, Z.},
  \bibinfo{author}{Pancorbo, I.}, \bibinfo{author}{Gurumurthy, K.M.},
  \bibinfo{author}{Freyermuth, V.}, \bibinfo{author}{Islam, E.},
  \bibinfo{author}{Vijayagopal, R.}, \bibinfo{author}{Stinson, M.},
  \bibinfo{author}{Rousseau, A.}, \bibinfo{year}{2021}.
\newblock \bibinfo{title}{A real-time energy and cost efficient vehicle route
  assignment neural recommender system}.
\newblock \bibinfo{journal}{arXiv preprint arXiv:2110.10887} .


\bibitem[{O’Keeffe et~al.(2019)O’Keeffe, Anjomshoaa, Strogatz, Santi and
  Ratti}]{o2019quantifying}
\bibinfo{author}{O’Keeffe, K.P.}, \bibinfo{author}{Anjomshoaa, A.},
  \bibinfo{author}{Strogatz, S.H.}, \bibinfo{author}{Santi, P.},
  \bibinfo{author}{Ratti, C.}, \bibinfo{year}{2019}.
\newblock \bibinfo{title}{Quantifying the sensing power of vehicle fleets}.
\newblock \bibinfo{journal}{Proceedings of the National Academy of Sciences}
  \bibinfo{volume}{116}, \bibinfo{pages}{12752--12757}.
  
  \bibitem[{Pacheco et~al.(2013)Pacheco, Caballero, Laguna and
  Molina}]{pacheco2013bi}
\bibinfo{author}{Pacheco, J.}, \bibinfo{author}{Caballero, R.},
  \bibinfo{author}{Laguna, M.}, \bibinfo{author}{Molina, J.},
  \bibinfo{year}{2013}.
\newblock \bibinfo{title}{Bi-objective bus routing: an application to school
  buses in rural areas}.
\newblock \bibinfo{journal}{Transportation Science} \bibinfo{volume}{47},
  \bibinfo{pages}{397--411}.
  
  
  \bibitem[{Rifki et~al.(2020)Rifki, Chiabaut and Solnon}]{rifki2020impact}
\bibinfo{author}{Rifki, O.}, \bibinfo{author}{Chiabaut, N.},
  \bibinfo{author}{Solnon, C.}, \bibinfo{year}{2020}.
\newblock \bibinfo{title}{On the impact of spatio-temporal granularity of
  traffic conditions on the quality of pickup and delivery optimal tours}.
\newblock \bibinfo{journal}{Transportation Research Part E: Logistics and
  Transportation Review} \bibinfo{volume}{142}, \bibinfo{pages}{102085}.
  
  
  
  \bibitem[{Sayarshad and Gao(2020)}]{sayarshad2020optimizing}
\bibinfo{author}{Sayarshad, H.R.}, \bibinfo{author}{Gao, H.O.},
  \bibinfo{year}{2020}.
\newblock \bibinfo{title}{Optimizing dynamic switching between fixed and
  flexible transit services with an idle-vehicle relocation strategy and
  reductions in emissions}.
\newblock \bibinfo{journal}{Transportation Research Part A: Policy and
  Practice} \bibinfo{volume}{135}, \bibinfo{pages}{198--214}.
  
  
  
  \bibitem[{Song et~al.(2020)Song, Han and Stettler}]{song2020deep}
\bibinfo{author}{Song, J.}, \bibinfo{author}{Han, K.},
  \bibinfo{author}{Stettler, M.E.}, \bibinfo{year}{2020}.
\newblock \bibinfo{title}{Deep-maps: Machine-learning-based mobile air
  pollution sensing}.
\newblock \bibinfo{journal}{IEEE Internet of Things Journal}
  \bibinfo{volume}{8}, \bibinfo{pages}{7649--7660}.
  
  
  
  \bibitem[{Tonekaboni et~al.(2020)Tonekaboni, Ramaswamy, Mishra, Setayeshfar and
  Omidvar}]{tonekaboni2020spatio}
\bibinfo{author}{Tonekaboni, N.H.}, \bibinfo{author}{Ramaswamy, L.},
  \bibinfo{author}{Mishra, D.}, \bibinfo{author}{Setayeshfar, O.},
  \bibinfo{author}{Omidvar, S.}, \bibinfo{year}{2020}.
\newblock \bibinfo{title}{Spatio-temporal coverage enhancement in drive-by
  sensing through utility-aware mobile agent selection}, in:
  \bibinfo{booktitle}{International Conference on Internet of Things},
  \bibinfo{organization}{Springer}. pp. \bibinfo{pages}{108--124}.
  
  
  
  \bibitem[{Wang et~al.(2018)Wang, Li, Qin, Wang and Li}]{wang2018maximizing}
\bibinfo{author}{Wang, C.}, \bibinfo{author}{Li, C.}, \bibinfo{author}{Qin,
  C.}, \bibinfo{author}{Wang, W.}, \bibinfo{author}{Li, X.},
  \bibinfo{year}{2018}.
\newblock \bibinfo{title}{Maximizing spatial--temporal coverage in mobile
  crowd-sensing based on public transports with predictable trajectory}.
\newblock \bibinfo{journal}{International Journal of Distributed Sensor
  Networks} \bibinfo{volume}{14}, \bibinfo{pages}{1550147718795351}.
  
  
  
  \bibitem[{Wang et~al.(2014)Wang, Birken and Shamsabadi}]{wang2014framework}
\bibinfo{author}{Wang, M.}, \bibinfo{author}{Birken, R.},
  \bibinfo{author}{Shamsabadi, S.S.}, \bibinfo{year}{2014}.
\newblock \bibinfo{title}{Framework and implementation of a continuous
  network-wide health monitoring system for roadways}, in:
  \bibinfo{booktitle}{Nondestructive Characterization for Composite Materials,
  Aerospace Engineering, Civil Infrastructure, and Homeland Security 2014},
  \bibinfo{organization}{SPIE}. pp. \bibinfo{pages}{107--118}.
  
  
  
  \bibitem[{Wang et~al.(2021)Wang, Wu, Wang, Zhen and Qu}]{wang2021emergency}
\bibinfo{author}{Wang, W.}, \bibinfo{author}{Wu, S.}, \bibinfo{author}{Wang,
  S.}, \bibinfo{author}{Zhen, L.}, \bibinfo{author}{Qu, X.},
  \bibinfo{year}{2021}.
\newblock \bibinfo{title}{Emergency facility location problems in logistics:
  Status and perspectives}.
\newblock \bibinfo{journal}{Transportation research part E: logistics and
  transportation review} \bibinfo{volume}{154}, \bibinfo{pages}{102465}.
  
  
  \bibitem[{Wu et~al.(2022)Wu, Lin, Liu and Jin}]{wu2022multi}
\bibinfo{author}{Wu, W.}, \bibinfo{author}{Lin, Y.}, \bibinfo{author}{Liu, R.},
  \bibinfo{author}{Jin, W.}, \bibinfo{year}{2022}.
\newblock \bibinfo{title}{The multi-depot electric vehicle scheduling problem
  with power grid characteristics}.
\newblock \bibinfo{journal}{Transportation Research Part B: Methodological}
  \bibinfo{volume}{155}, \bibinfo{pages}{322--347}.
  
  
  \bibitem[{Xu et~al.(2019)Xu, Chen, Pi, Joe-Wong, Zhang and Noh}]{xu2019ilocus}
\bibinfo{author}{Xu, S.}, \bibinfo{author}{Chen, X.}, \bibinfo{author}{Pi, X.},
  \bibinfo{author}{Joe-Wong, C.}, \bibinfo{author}{Zhang, P.},
  \bibinfo{author}{Noh, H.Y.}, \bibinfo{year}{2019}.
\newblock \bibinfo{title}{ilocus: Incentivizing vehicle mobility to optimize
  sensing distribution in crowd sensing}.
\newblock \bibinfo{journal}{IEEE Transactions on Mobile Computing}
  \bibinfo{volume}{19}, \bibinfo{pages}{1831--1847}.
  
  
  \bibitem[{Zhou et~al.(2020)Zhou, Liu, Wei and Golub}]{zhou2020bi}
\bibinfo{author}{Zhou, Y.}, \bibinfo{author}{Liu, X.C.}, \bibinfo{author}{Wei,
  R.}, \bibinfo{author}{Golub, A.}, \bibinfo{year}{2020}.
\newblock \bibinfo{title}{Bi-objective optimization for battery electric bus
  deployment considering cost and environmental equity}.
\newblock \bibinfo{journal}{IEEE Transactions on Intelligent Transportation
  Systems} \bibinfo{volume}{22}, \bibinfo{pages}{2487--2497}.
  
    
  
  
\end{thebibliography}

\end{document}